\documentclass[12pt,oneside,english]{amsart}
\usepackage[T1]{fontenc}
\usepackage[utf8]{inputenc}
\usepackage{geometry}
\usepackage{fancyhdr}
\usepackage{amstext}
\usepackage{amsthm}
\usepackage{amssymb}
\usepackage{mathdots}
\usepackage{esint}
\usepackage{subcaption}
\usepackage{caption}
\usepackage{float} 
\usepackage{placeins} 

\usepackage[normalem]{ulem}
\usepackage[dvipsnames]{xcolor}
\definecolor{editbrown}{rgb}{0.45,0.22,0.05}

\usepackage{dsfont}

\makeatletter
\providecommand{\tabularnewline}{\\}
\numberwithin{equation}{section}
\numberwithin{figure}{section}
\theoremstyle{plain}
\newtheorem{thm}{\protect\theoremname}
\newtheorem{prop}{\protect\propositionname}
\newtheorem{lemma}{\protect\lemmaname}
\newtheorem{corollary}{\protect\corollaryname}
\theoremstyle{remark}

\theoremstyle{remark}
\newtheorem*{remark*}{Remark}

\theoremstyle{definition}
\newtheorem{example}{Example}
\usepackage[english]{babel}
\usepackage{fancyhdr}
\usepackage{tikz}
\usepackage{graphicx}
\usepackage{hyperref}

\DeclareGraphicsExtensions{.pdf,.png,.jpg,.jpeg}

\newif\ifRplotFound
\newcommand{\RplotInclude}[1]{\includegraphics[width=\linewidth]{#1}}%
\newcommand{\Rplot}[1]{%
  \begingroup
  \RplotFoundfalse
  \def\RplotTry##1{%
    \IfFileExists{##1}{\RplotInclude{##1}\RplotFoundtrue}{}%
    \ifRplotFound\else\IfFileExists{##1.png}{\RplotInclude{##1.png}\RplotFoundtrue}{}\fi%
    \ifRplotFound\else\IfFileExists{##1.pdf}{\RplotInclude{##1.pdf}\RplotFoundtrue}{}\fi%
    \ifRplotFound\else\IfFileExists{##1.jpg}{\RplotInclude{##1.jpg}\RplotFoundtrue}{}\fi%
    \ifRplotFound\else\IfFileExists{##1.jpeg}{\RplotInclude{##1.jpeg}\RplotFoundtrue}{}\fi%
  }%
  \RplotTry{#1}%
  \ifRplotFound\else\RplotTry{fig/#1}\fi%
  \ifRplotFound\else\RplotTry{figures/#1}\fi%
  \ifRplotFound\else\RplotTry{plots/#1}\fi%
  \ifRplotFound\else\RplotTry{images/#1}\fi%
  \ifRplotFound\else\RplotTry{img/#1}\fi%
  \ifRplotFound\else\RplotTry{v63_package/#1}\fi%
  \ifRplotFound\else\RplotTry{./v63_package/#1}\fi%
  \ifRplotFound\else
    \fbox{\ttfamily Missing figure file: \detokenize{#1}}%
  \fi
  \endgroup
}
\graphicspath{{./}{./fig/}{./figures/}{./plots/}{./images/}{./img/}{./v63_package/}{v63_package/}}
\usetikzlibrary{arrows,shapes}
\usetikzlibrary{calc,decorations.markings}
\tikzset{
  reversed with radius/.style={
    x radius=#1,
    y radius=-#1,
 }
}
\tikzset{
  with arrows/.style={
    decoration={ markings,
      mark=between positions #1 and .999 step #1 with {\arrow{stealth}}
    }, postaction={decorate}
  }, with arrows/.default=25mm,
}
\providecommand{\tabularnewline}{\\}

\providecommand{\lemmaname}{Lemma}
\providecommand{\propositionname}{Proposition}
\providecommand{\theoremname}{Theorem}
\providecommand{\corollaryname}{Corollary}
\newcommand{\abs}[1]{\ensuremath{|#1|}}

\newcommand{\diff}{\mathop{}\!\mathrm{d}}

\newcommand{\cF}{\mathcal{F}}

\newcommand{\cO}{\mathcal{O}}

\newcommand{\cS}{\mathcal{S}}

\providecommand{\dd}{\mathop{}\!\mathrm{d}}
\providecommand{\corollaryname}{Corollary}
\providecommand{\lemmaname}{Lemma}
\providecommand{\propositionname}{Proposition}
\providecommand{\remarkname}{Remark}
\providecommand{\theoremname}{Theorem}
\theoremstyle{definition}
\newtheorem{defn}{Definition}[section]
\newcounter{appsection}
\newcounter{appsubsection}[appsection]
\renewcommand{\theappsection}{\arabic{appsection}}
\renewcommand{\theappsubsection}{\theappsection.\arabic{appsubsection}}
\newcommand{\appsection}[2]{\refstepcounter{appsection}\section*{Appendix~\theappsection: #1}\label{#2}}
\newcommand{\appsubsection}[2]{\refstepcounter{appsubsection}\subsection*{\theappsubsection. #1}\label{#2}}

\DeclareFontFamily{U}{wncy}{}
\DeclareFontShape{U}{wncy}{m}{n}{<->wncyr10}{}
\DeclareSymbolFont{mcy}{U}{wncy}{m}{n}
\DeclareMathSymbol{\Sh}{\mathord}{mcy}{"58}

\makeatother
\begin{document}
\title{Projected Inner-Function Dynamics and Crystalline Measures of Meyer-Blaschke type}
\author{Oleg Szehr}
\address{Dalle Molle Institute for Artificial Intelligence (IDSIA) - SUPSI/USI\\
Via la Santa 1, 6962 Lugano-Viganello, Switzerland}
\email{oleg.szehr@idsia.ch}

\author{Rachid Zarouf}
\address{Aix Marseille Univ, Lab ADEF, Campus Univ St Jerome, 52 Ave Escadrille Normandie\\
F-13013 Marseille, France}
\address{Aix-Marseille Univ, Universit\'{e} de Toulon, CNRS, CPT, Marseille, France}
\email{rachid.zarouf@univ-amu.fr}

\begin{abstract}
We generalize a construction of Yves Meyer of sparse crystalline measures arising from powers of a Blaschke factor. Starting from a recursion $f_n=\theta^n f_0$ on the unit circle where \(\theta\) is an inner function, we project the Fourier coefficient array $\widehat{f_n}(k)$ to the real line by placing its entries at the frequencies $k+\alpha n$. We identify the role of model spaces in this construction: in Meyer's one-factor Blaschke recursion, the requirement that the coefficient array $\widehat{f_n}(k)$ vanish whenever \(kn<0\) is equivalent to $f_0\in K_{zb_\lambda}$, and for general inner functions the condition $f_0\in K_{z\theta}$ yields a purely atomic Radon measure with locally finite support and polynomial growth on the Fourier side. We also show that, when $f_0$ is holomorphic in an annulus containing the unit circle, exponential Fourier decay is sufficient to obtain a purely atomic Radon measure of polynomial growth, though not necessarily locally finite support. For finite Blaschke products, the coefficient recursion gives an explicit annihilating exponential polynomial whose zero set controls the support and separation of the inverse Fourier transform. This yields Meyer-Blaschke-type crystalline measures and Poisson identities with sampling and finite-truncation consequences.
\end{abstract}

\maketitle
\section{Introduction}
The purpose of this paper is to build on Y.~Meyer's recent construction of sparse crystalline measures arising from powers of a Blaschke factor, identify its underlying mechanism, and develop both its structural and quantitative aspects. We shall refer to the measures produced by this mechanism and its variants as \emph{Meyer-Blaschke-type measures}. Meyer's article is one of those characteristically economical pieces in which a simple lattice recurrence hides a rich Fourier-theoretic object. Our aim is to isolate the elegant argument behind Meyer's construction, clarify the role of the initial condition, extend the construction from one Blaschke factor to finite Blaschke products and more general inner functions, and prove quantitative estimates that lead to finite-truncation results.

The basic object is obtained as follows. Starting from a sequence of functions $f_n$ on the unit circle, one forms the coefficient array
\[
c(k,n)=\widehat{f_n}(k),
\qquad k,n\in\mathbb Z,
\]
and projects this two-dimensional array to the real line by placing $c(k,n)$ at the frequency $k+\alpha n$. This gives the formal Fourier-side measure
\[
\widehat\kappa
=
\sum_{k,n\in\mathbb Z}
c(k,n)\delta_{k+\alpha n}.
\]
In Meyer's original construction one takes $f_n=b_\lambda^nf_0$, where
\[
b_\lambda(z)=\frac{z-\lambda}{1-\overline\lambda z}
\]
is the Blaschke factor associated with \(\lambda\) in the unit disk, $f_0$ denotes an initial condition and $\alpha$ is irrational. The central point is that the recurrence satisfied by $f_n$ becomes, after projection, an annihilation relation for $\widehat\kappa$. Taking inverse Fourier transforms then constrains the support of $\kappa$ to lie in the zero set of an explicit exponential polynomial. This mechanism is the source of discreteness, and ultimately of sparsity, for the Meyer-Blaschke-type measures.

Our main point is that Meyer's construction is an instance of a more general principle: inner-function recursions on the unit circle can generate purely atomic measures on the Fourier side and, under additional structural hypotheses, sparse crystalline measures. We begin by making precise which initial conditions $f_0$ are compatible with this mechanism for the recursion $f_n=\theta^n f_0$, where \(\theta\) is an inner function. In the case $\theta=b_\lambda$, requiring the coefficient array $c(k,n)$ to vanish whenever \(kn<0\) is equivalent to the model-space condition, namely \(f_0\in K_{zb_\lambda}\).
The same link persists for arbitrary inner functions, where $f_0\in K_{z\theta}$ provides the natural class of initial conditions for which the projected Fourier-side measure has locally finite support and polynomial growth.

We then prove a complementary stability principle: the exact condition \(c(k,n)=0\) whenever \(kn<0\) may be replaced by exponential Fourier decay. For initial conditions holomorphic in an annulus around the unit circle, the projected Fourier-side series still defines a purely atomic Radon measure of polynomial growth, hence a tempered distribution, although its support need not be locally finite.

For finite Blaschke products, these mechanisms combine to give a direct extension of Meyer's Blaschke-factor construction. The recursion becomes a finite recurrence relation for the coefficients $c(k,n)$, which produces an explicit annihilating exponential polynomial on the inverse Fourier side. Its zero set, controlled through the phase derivative of the Blaschke product, gives quantitative separation estimates for the support of $\kappa$. In the cases where the Fourier-side support is locally finite, this yields sparse Meyer-Blaschke-type crystalline measures.

We finally turn to the sampling interpretation of the Poisson identities associated with Meyer-Blaschke-type measures. The generalized Poisson formula relates samples of a test function on the support of $\kappa$ to values of its Fourier transform on the support of $\widehat\kappa$. In the presence of a suitable spectral gap, this relation can be read as a reconstruction formula for band-limited functions. For general Schwartz functions, it instead leads naturally to finite Fourier-side approximations, whose convergence is quantified by the truncation estimates proved below.

In the broader circle of ideas connecting sampling with model spaces, important work based on a different mechanism is due to
A.~D.~Baranov, P.~Jaming, K.~Kellay, and M.~Speckbacher
~\cite{BaranovJamingKellaySpeckbacher2024}, who extend oversampling phenomena from the Paley--Wiener setting to model spaces associated with meromorphic inner functions. Related quantitative estimates for sampling constants in model spaces were obtained by A.~Hartmann, P.~Jaming, and K.~Kellay~\cite{HartmannJamingKellay2020}.

The paper is organized as follows. Section~\ref{sec:prelim} recalls the distributional framework for crystalline measures and reviews the Blaschke-product construction underlying Meyer's examples. Section~\ref{sec:meyer-measures} develops the general Meyer--Blaschke-type mechanism: first for one Blaschke factor, then for general inner functions, initial conditions holomorphic in an annulus, and finite Blaschke products. Section~\ref{sec:fourier-trunc-sampling} proves the Fourier-side truncation estimates and explains their sampling interpretation. The appendices collect the asymptotic and summability estimates for coefficients of powers of Blaschke products used in the main text.

\section{Preliminaries and Background}\label{sec:prelim}
\subsection{Crystalline measures}
We begin by recalling the measure-theoretic and distributional framework; see~\cite{Bourbaki} and~\cite{HormanderI} for background. A tempered distribution is a continuous linear functional on the Schwartz space, and a Radon measure defines such a distribution provided it has at most polynomial growth at infinity. We write \(\delta_\lambda\) for the Dirac measure at \(\lambda\in\mathbb R\), and for countable $\Lambda\subset\mathbb R$ we shall consider formal series of the form
\[
\mu=\sum_{\lambda\in\Lambda} c(\lambda)\delta_\lambda .
\]
Such a series defines a Radon measure if
$\sum_{|\lambda|<R} |c(\lambda)|<\infty$ for every $R>0$. In that case \(\mu\) is a purely atomic measure, and its support is the closure of the set of \(\lambda\in\Lambda\) for which \(c(\lambda)\neq 0\). The set \(\Lambda\subset\mathbb R\) is {locally finite} if its intersection with every bounded set is finite, and {uniformly discrete} if the distance between any two distinct points of \(\Lambda\) is bounded below by a positive constant. The Dirac comb \(\Sh_{a,b}=\sum_{\lambda\in a\mathbb Z+b}\delta_\lambda\) is a purely atomic measure with locally finite, uniformly discrete support; it has at most polynomial growth at infinity and thus defines a tempered distribution. We use the Fourier transform convention
\[
\widehat f(\xi)=\int_{\mathbb R} f(x)e^{-2\pi i x\xi}\,\diff x.
\]
Then the distributional Fourier transform of the Dirac comb $\Sh_{a,0}$, $a>0$, is the Dirac comb $\widehat{\Sh}_{a,0}=\frac{1}{a}\Sh_{1/a,0}$. To translate Dirac combs, we use convolution on Schwartz functions: if \(u\in\cS(\mathbb R)\) and \(v\in\cS'(\mathbb R)\), then \(u*v\) is defined by duality, and if \(u,v\in\cS'(\mathbb R)\) with one factor compactly supported, then \(u*v\) is again well defined. In these cases the convolution theorem asserts \(\cF(u*v)=\cF(u)\,\cF(v)\). Since \(\Sh_{a,b}=\Sh_{a,0}*\delta_b\), taking Fourier transforms and using \(\widehat{\delta_b}(\xi)=e^{-2\pi i b\xi}\) yields the classical Poisson summation formula
\[
\widehat{\Sh_{a,b}}(\xi)=e^{-2\pi i b\xi}\,\frac{1}{a}\Sh_{1/a,0}(\xi),\ \text{that is},\ \sum_{k\in\mathbb Z}\widehat\varphi(ak+b)
=
\frac{1}{a}\sum_{k\in\mathbb Z}
e^{-2\pi i b k/a}\,{\varphi}\!\left(\frac{k}{a}\right),
\ \varphi\in\mathcal S(\mathbb R).
\]
Analogously, if
\[
\kappa=\sum_{\lambda\in\Lambda}c(\lambda)\delta_\lambda
\qquad\text{and}\qquad
\widehat\kappa=\sum_{s\in S}d(s)\delta_s,
\]
then applying the distributional Fourier transform to test functions yields the generalized Poisson identity
\[
\sum_{\lambda\in\Lambda}c(\lambda)\widehat\varphi(\lambda)
=
\sum_{s\in S}d(s)\varphi(s).
\]
This motivates the following definition.
\begin{defn}\label{def:crystallMeath}
A \emph{crystalline measure} on \(\mathbb R\) is a purely atomic Radon measure \(\kappa\) of the form $\kappa=\sum_{\lambda\in\Lambda} c_\lambda\,\delta_\lambda$, where $\Lambda$ is locally finite, such that \(\kappa\) defines a tempered distribution and its Fourier transform \(\widehat{\kappa}\) is a distribution that again corresponds to a purely atomic Radon measure, say $\widehat{\kappa}=\sum_{s\in S} d_s\,\delta_s$, with locally finite \(S\). A crystalline measure $\kappa$ is called {sparse} if its support \(\Lambda\) is uniformly discrete.
\end{defn}

Related distributional summation formulas can already be found in the 1950s literature on number theory and harmonic analysis; see A.~Weil~\cite{Weil1952}, J.~-P.~Kahane and S.~Mandelbrojt~\cite{KahaneMandelbrojt1958}, and A.P.~Guinand~\cite{Guinand1959}. The modern theory of crystalline measures was subsequently revived and developed in a systematic way through a number of contributions; see, for example, A.~Córdoba~\cite{Cordoba1989}, Y.~Meyer~\cite{MeyerPNAS16,MeyerRMI2017}, N.~Lev and A.~Olevskii~\cite{LevOlevskii2013,LevOlevskii2015,LevOlevskii2016}, M.~Kolountzakis~\cite{Kolountzakis2016}, as well as the survey~\cite{MeyerBourbaki1194}. 
More recent contributions include the Fourier interpolation formula of D.~Radchenko and M.~Viazovska~\cite{RadchenkoViazovska2019}, work on Fourier interpolation and uniqueness pairs by A.~Kulikov, F.~Nazarov and M.~Sodin~\cite{Kulikov2021,KNS2023}, the zeta/L-function interpolation framework of A.~Bondarenko, D.~Radchenko, and K.~Seip~\cite{BondarenkoRadchenkoSeip2023}, the powers-of-integers uniqueness pairs of J.~P.~G.~Ramos and M.~Sousa~\cite{RamosSousa2022}. A signal-sampling perspective on this circle of ideas, motivated by engineering applications and based on two-sided sampling in the time and Fourier domains, is developed in~\cite{Szehr2025,KayaalpSzehr2026}. For a recent historical account of this circle of ideas, see Meyer~\cite{MeyerKahane2026}.

Recent related developments include the construction by P.~Kurasov and P.~Sarnak of positive crystalline measures and Fourier quasicrystals from stable polynomial pairs~\cite{KurasovSarnak2020}, the classification framework of F.~Gon\c{c}alves for Fourier summation formulae and crystalline measures with quadratic decay~\cite{Goncalves2026}, and the work of J.~Maz\'a\v{c}, C.~Richard and N.~Strungaru on almost periodicity and translation boundedness for crystalline measures~\cite{MazacRichardStrungaru2026}. These results provide part of the recent structural background for the Meyer--Blaschke-type constructions considered here.

In a different, multidimensional direction, Meyer constructed crystalline measures from Ahern measures and products of independent lighthouses~\cite{MeyerMultidimensional2023}. In a 2026 addendum to that article, communicated privately to us~\cite{MeyerPrivate2026}, Meyer notes that the one-dimensional counterpart of one of the results in that paper had already been proved by D.~N.~Clark~\cite{Clark1972}. Related recent work of Meyer develops generalized inner functions on the \(n\)-dimensional torus, including examples built from Blaschke products, and relates them to Aleksandrov--Clark measures and lighthouses~\cite{MeyerKolmogorov}.

\subsection{Blaschke products and Meyer's sparse crystalline measures}

Since we build on Meyer's work, we briefly recall the underlying setting. 

Let \(\mathbb D\subset\mathbb C\) be the unit disk, \(\mathbb T=\partial\mathbb D\) its boundary, and \(H^2(\mathbb D)\) the Hardy space of holomorphic functions on \(\mathbb D\) whose radial boundary values belong to \(L^2(\mathbb T)\). We write \(\widehat f(k)=\int_{0}^{1} f(e^{2\pi ix})e^{-2\pi ikx}\text{d} x\), \(k\in\mathbb Z\), for the Fourier coefficients of \(f\in L^2(\mathbb T)\). The Taylor coefficients of \(f\in H^2(\mathbb D)\) agree with the Fourier coefficients of its radial boundary in $L^2(\mathbb T)$ and we write $f=\sum_{k\geq0}\hat{f}(k)z^k$. Similarly $H^\infty(\mathbb D)$ is the Hardy space of holomorphic functions whose boundary values belong to $L^\infty(\mathbb T)$. A function $\theta\in H^\infty(\mathbb D)$ is called inner if $|\theta(z)|=1$ almost everywhere on $\mathbb T$. The multiplication by inner functions is an isometry of \(H^2(\mathbb D)\). The model space is $K_\theta = H^2(\mathbb D)\ominus\theta H^2(\mathbb{D})$, where the underlying inner product is taken from $L^2(\mathbb{T})$, see~\cite{NikolskisBook} for an introduction to model spaces.

For \(\lambda\in\mathbb D\), the elementary Blaschke factor
\[
b_\lambda(z)=\frac{z-\lambda}{1-\overline\lambda z}
\]
is an automorphism of \(\mathbb D\) satisfying \(|b_\lambda(z)|=1\) on \(\mathbb T\). Moreover, if \(n>0\), then \(b_\lambda^n\) is holomorphic and \(\widehat{b_\lambda^n}(k)=0\) for \(k<0\); similarly, for \(n<0\), one has \(\widehat{b_\lambda^n}(k)=0\) for \(k>0\), and \(\widehat{b_\lambda^n}(k)=\overline{\widehat{b_\lambda^{|n|}}(-k)}\).

We begin with a prototypical instance of Meyer's construction of sparse crystalline measures, based on the coefficient array generated by iterates of a single Blaschke factor. Given \(f_0\in H^2(\mathbb D)\), consider \(f_n=b_\lambda^n f_0\) as an element of $L^2(\mathbb T)$ and the array $c$ with coefficients \(c(k,n)=\widehat{f_n}(k)\), $k,n\in\mathbb Z$. The relation \(f_{n+1}=b_\lambda f_n\), equivalently \((1-\bar\lambda z)f_{n+1}=(z-\lambda)f_n\), induces the recursion
\begin{align}
c(k+1,n+1)-\bar\lambda\,c(k,n+1)=c(k,n)-\lambda\,c(k+1,n),
\qquad k,n\in\mathbb Z,\label{eq:simpleRecursion}
\end{align}
with initial condition \(c(k,0)=\widehat f_0(k)\). Observe that whenever \(kn<0\) we have $\widehat{b_\lambda^n}(k)=0$. By contrast, $c$ may assume nonzero values for any $k,n\in\mathbb Z$. We will show in the main body that $c(k,n)=0$ for $kn<0$ if and only if
$f_0$ lies in the {model space} $K_{zb_\lambda}=H^2(\mathbb D)\ominus z b_\lambda H^2(\mathbb D)$, see~Lemma~\ref{lem:model-space-characterization} below.
%
%

For irrational \(\alpha\) consider the formal Dirac series
\[
\widehat\kappa=\sum_{n,k\in\mathbb Z}c(k,n)\delta_{k+\alpha n}.
\]

Whether this formal series defines a Radon measure, a tempered distribution, and eventually the Fourier transform of a crystalline measure depends on three key ingredients:

\emph{1)} If $\alpha>0$ and $c(k,n)=0$ whenever $kn<0$, then $|k+\alpha n|=|k|+\alpha |n|$. Thus only finitely many pairs $(k,n)$ contribute to a compact interval and, consequently, the support of $\hat\kappa$ is locally finite. 

\emph{2)} If, in addition, the coefficients are uniformly bounded, then $\widehat\kappa$ is a purely atomic Radon measure and satisfies the polynomial growth estimate $|\widehat\kappa|([-R,R])=O(R^2)$, hence defines a tempered distribution. 

\emph{3)} Finally, one must show that the inverse Fourier transform $\kappa$ is again a purely atomic Radon measure, not merely a tempered distribution. The finite recursion~\eqref{eq:simpleRecursion} provides the geometric support constraint needed to obtain this.

As a consequence, Meyer shows that the choice $f_0=1$, \(\lambda=\frac12\) and \(\alpha=\sqrt2\) yields the sparse crystalline measure
\[
\widehat\kappa_0=\sum_{n,k\in\mathbb Z}\widehat{b_{1/2}^n}(k)\,\delta_{k+\sqrt2\,n}.
\]

While the structural conditions establish crystallinity, finer estimates for $\widehat{b_\lambda^n}(k)$ are needed to quantify the Fourier-side masses. Meyer used the global bound
$$
\|\widehat{b_\lambda^n}\|_{\ell^\infty}\le C n^{-1/3},
\qquad n\ge1,$$
which captures the maximal order of magnitude of the coefficients. More recent work on the asymptotic analysis of the Fourier coefficients \(\widehat{b_\lambda^n}(k)\)~\cite{SzehrZaroufJAM,SzehrZaroufJAT2022,BFZCA} allows for a refined study of the relationship between initial conditions and decay estimates in Meyer's construction. Speaking roughly, writing \(\alpha_0=\frac{1-\lambda}{1+\lambda}\), the coefficients exhibit three qualitatively distinct regimes as \(n\to\infty\): exponential decay when \(k/n\) stays outside the critical interval \([\alpha_0,\alpha_0^{-1}]\), oscillatory decay of order \(n^{-1/2}\) in its interior, and Airy-type transition behavior of order \(n^{-1/3}\) near the interval's endpoints \(\alpha_0\) and \(\alpha_0^{-1}\). This is illustrated in Figure~\ref{POU}; the detailed region-by-region expansions are presented in Appendix~\ref{app:asymptotic-formulas}.

\begin{figure}[ht!]
\centering
\Rplot{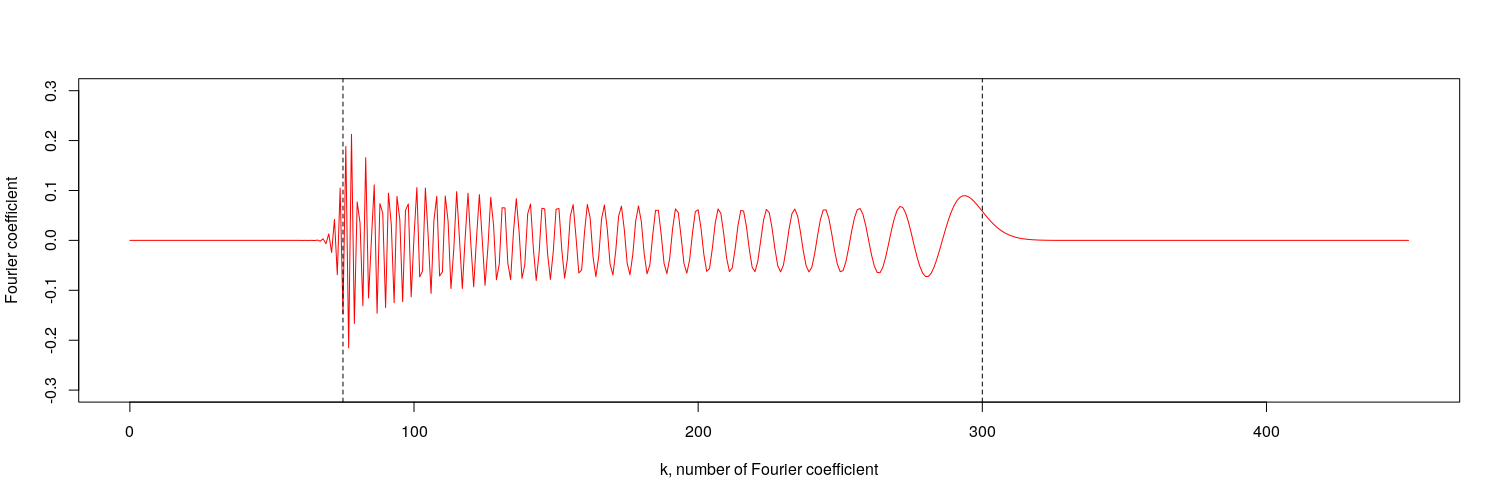}
\caption{Numerical plot of the Fourier coefficients $\widehat{b^n_\lambda}(k)$ for $\lambda=\frac{1}{3}$ and $n=150$. Outside the critical intervals $[\frac{1-\lambda}{1+\lambda}n,\frac{1+\lambda}{1-\lambda}n]$ (marked by dashed, vertical lines) the coefficients decay exponentially.}
\label{POU} 
\end{figure}

\section{Crystalline measures of Meyer-Blaschke type}\label{sec:meyer-measures}

The key structural point in Meyer's construction is that a recursion for a coefficient array on $\mathbb Z^2$ is converted into an annihilation relation for an associated atomic distribution on $\mathbb R$, and subsequently into a geometric constraint on the support of its inverse Fourier transform:
\[
c*a=0
\quad\Longrightarrow\quad
\widehat\kappa*\tau=0
\quad\Longrightarrow\quad
\kappa\,\phi=0,
\qquad \phi=\cF^{-1}(\tau).
\]
The first identity is the coefficient recursion written as a convolution equation on \(\mathbb Z^2\), where the sequence $a$ has finite support. After placing the coefficient \(c(k,n)\) at the point \(k+\alpha n\), this becomes the convolution relation \(\widehat\kappa*\tau=0\) on \(\mathbb R\), where \(\tau\) is the finite atomic measure determined by $a$. Taking inverse Fourier transforms then gives the pointwise annihilation relation \(\kappa\,\phi=0\). It follows that the support of $\kappa$ is contained in the zero set of $\phi$,
\[
\text{supp}(\kappa)\subset Z(\phi)=\{x\in\mathbb R:\phi(x)=0\}.
\]
Hence the support of \(\kappa\) is controlled by the real zero set of the explicit function \(\phi\). In particular, local finiteness of $Z(\phi)$ forces \(\kappa\) to be supported on a discrete set - if the relevant zeros are simple the distribution-structure theorem implies $\kappa$ is purely atomic - and uniform discreteness of $Z(\phi)$ yields sparsity.

We first make this mechanism explicit for the coefficient arrays generated by the
one-factor recursion \(b_\lambda^n f_0\). We then turn to extending Meyer's work.
We describe sufficient conditions on the initial condition \(f_0\) for the recursion to
yield a crystalline measure. 
%
%
%
%
%
%

\subsection{Crystalline measures generated by powers of a  Blaschke factor}

This section serves to recapitulate Meyer's construction and to illustrate our ideas in a well-known setting. 
We formulate coefficient-array conditions arising from the recursion~\eqref{eq:simpleRecursion} that ensure that the associated projected Dirac series defines a sparse crystalline measure, characterize the admissible initial condition \(f_0\) in the Hardy-space setting, and provide a quantitative separation bound for \cite[Lemmas~6--7]{Meyer}.
\begin{thm}\label{thm:meyer}
Let $\lambda\in\mathbb D$, let \(\alpha>0\) be irrational and let \(c:\mathbb Z^2\to\mathbb C\) satisfy
\[
c(k+1,n+1)-\bar\lambda\,c(k,n+1)=c(k,n)-\lambda\,c(k+1,n),
\quad k,n\in\mathbb Z.
\]
Assume moreover that $c(k,n)=0$ whenever $kn<0$ and that $\sup_{k,n\in\mathbb Z}|c(k,n)|<\infty$.
Then the formal series
\[
\widehat\kappa=\sum_{k,n\in\mathbb Z} c(k,n)\,\delta_{k+\alpha n}
\]
defines a purely atomic Radon measure with locally finite support and at most quadratic growth, hence a tempered distribution. Its inverse Fourier transform \(\kappa\) is a sparse crystalline measure, whose support \(\Lambda=\text{supp}(\kappa)\) is \(\delta(\alpha,\lambda)\)-separated, where
\[
\delta(\alpha,\lambda)=
\frac{2(1+\alpha)(1-|\lambda|)}
{2\pi\bigl((1+\alpha^2)(1+|\lambda|)+2\alpha\bigr)}.
\]
\end{thm}
\begin{proof}[Proof of Theorem~\ref{thm:meyer}]
The coefficient recursion in eq.~\eqref{eq:simpleRecursion} is of the form
\[
c*a=0,
\]
with
\[
a(0,0)=1,\qquad a(1,0)=-\bar\lambda,\qquad a(0,1)=\lambda,\qquad a(1,1)=-1,
\]
and \(a(k,n)=0\) otherwise. Projecting each $(k,n)\in\mathbb Z^2$ to $k+\alpha n\in\mathbb R$, yields a formal Dirac series $c\mapsto\widehat{\kappa}=\sum_{k,n}c(k,n)\delta_{k+\alpha n}$ and a finite atomic measure $a\mapsto\tau=\sum_{k,n}a(k,n)\delta_{k+\alpha n}$. In other words the recursion becomes
\[
\widehat\kappa * \tau=0,
\qquad
\tau=\delta_0-\bar\lambda\,\delta_1+\lambda\,\delta_\alpha-\delta_{1+\alpha},
\]
and after inverse Fourier transform,
\[
\kappa\,\phi=0,
\qquad
\phi(x)=1-\bar\lambda e^{2\pi i x}+\lambda e^{2\pi i\alpha x}-e^{2\pi i(1+\alpha)x}.
\]
On the support of \(c\), $|k+\alpha n|=|k|+\alpha|n|$ 
and thus
\[
\sum_{|k+\alpha n|\le R}|c(k,n)|
\le
\sup_{k,n\in\mathbb Z}|c(k,n)|\,\cdot
\#\bigl\{(k,n)\in\mathbb Z^2:\ |k|+\alpha|n|\le R\bigr\}
=O(R^2).
\]
Thus \(\widehat\kappa\) assigns finite measure to bounded intervals and is thus a Radon measure. It has at most polynomial growth and consequently defines a tempered distribution. 

Since $\kappa\phi=0$, we have $\text{supp}(\kappa)\subset Z(\phi)$. The function $\phi$ is a nonzero exponential polynomial, hence real analytic and not identically zero; therefore $Z(\phi)$ is locally finite. Thus $\kappa$ is a tempered distribution supported on a discrete set. By the standard structure theorem for distributions supported on a discrete set $\kappa$ is locally a finite linear combination of Dirac masses and their derivatives; see~\cite{HormanderI}. 
Computing derivatives we find
\[
\phi'(x)=2\pi i\Bigl(-\bar\lambda e^{2\pi i x}
+\alpha\lambda e^{2\pi i\alpha x}
-(1+\alpha)e^{2\pi i(1+\alpha)x}\Bigr),
\]
\[
\phi''(x)=(2\pi)^2\Bigl(\bar\lambda e^{2\pi i x}
-\alpha^2\lambda e^{2\pi i\alpha x}
+(1+\alpha)^2 e^{2\pi i(1+\alpha)x}\Bigr),
\]
so that
\[
\min_{x\in\mathbb R}|\phi'(x)|\ge 2\pi(1+\alpha)(1-|\lambda|)>0,
\qquad
\|\phi''\|_{L^\infty(\mathbb R)}
\le (2\pi)^2\bigl((1+\alpha^2)(1+|\lambda|)+2\alpha\bigr).
\]
Since $\min_{x\in\mathbb R}|\phi'(x)|$ is strictly separated from $0$ no derivative terms can occur in $\kappa$. Thus $\kappa$ is a crystalline measure. 

To estimate the separation of $\Lambda$, note that if \(x_0<x_1\) are consecutive zeros of \(\phi\), then Taylor's formula at \(x_0\) and the identities \(\phi(x_0)=\phi(x_1)=0\) give
\[
x_1-x_0\ge \frac{2\,\min_{\mathbb R}|\phi'|}{\|\phi''\|_{L^\infty(\mathbb R)}}
\ge
\frac{2(1+\alpha)(1-|\lambda|)}
{2\pi\bigl((1+\alpha^2)(1+|\lambda|)+2\alpha\bigr)}.
\]
Hence \(\Lambda\) is uniformly discrete, so \(\kappa\) is sparse.
\end{proof}
We now turn to the Hardy space setting to demonstrate that the conditions of the theorem can be met, characterizing the necessary initial conditions for the {recursion associated with a single Blaschke factor}.
\begin{lemma}\label{lem:model-space-characterization}
Let \(\lambda\in\mathbb D\), let \(f_0\in H^2(\mathbb D)\), and for \(n\in\mathbb Z\) define
\[
f_n=b_\lambda^n f_0
\]
via boundary values on \(\mathbb T\). Set $c(k,n)=\widehat{f_n}(k)$, $k,n\in\mathbb Z$. Then the coefficient array \(c\) satisfies~\eqref{eq:simpleRecursion} and the following are equivalent:
\begin{enumerate}
\item $c(k,n)=0$ \text{whenever} $kn<0$,
\item $f_0\in K_{zb_\lambda}=H^2(\mathbb D)\ominus z b_\lambda H^2(\mathbb D)$.
\end{enumerate}
\end{lemma}

\begin{proof}
\noindent\((2)\Rightarrow(1)\).
Assume \(f_0\in K_{zb_\lambda}\). This means that for every \(k\ge0\) we have
        $0
        =
        \langle f_0, z b_\lambda z^k\rangle_{L^2(\mathbb T)}
        =
        \widehat{\overline{b_\lambda}f_0}(k+1)$. In other words \(\bar b_\lambda f_0\) has no Fourier coefficients at strictly positive frequencies. Since \(\bar b_\lambda\) has no Fourier coefficients at strictly positive frequencies, it follows inductively that for every \(m\ge 1\) that \(\widehat{\bar b_\lambda^{\,m}f_0}(k)=0\) for all \(k>0\). Equivalently, 
\[
c(k,-m)=0,\quad k>0,\ m\ge 1.
\]
On the other hand, for \(n\ge 0\) one has \(f_n=b_\lambda^n f_0\in H^2(\mathbb D)\), hence
\[
c(k,n)=0,\quad k<0,\ n\ge 0.
\]
Therefore \(kn<0\) implies \(c(k,n)=0\).

\smallskip
\noindent\((1)\Rightarrow(2)\).
Assume that \(c(k,n)=0\) whenever \(kn<0\). Taking \(n=-1\), we obtain
\[
\widehat{\bar b_\lambda f_0}(k)=c(k,-1)=0,\quad k>0.
\]
This means that \(\overline{b_\lambda}f_0\) has no Fourier coefficients at strictly
positive frequencies and $\langle f_0,zb_\lambda z^{k-1}\rangle_{L^2(\mathbb T)}
        =
        \widehat{\overline{b_\lambda}f_0}(k)
        =
        0$.
Thus \(f_0\) is orthogonal to \(zb_\lambda H^2(\mathbb D)\) and 
$f_0\in H^2(\mathbb D)\ominus zb_\lambda H^2(\mathbb D)=K_{zb_\lambda}$.

\end{proof}
This leads to the following immediate corollary that extends Meyer's {construction for a single Blaschke factor} \cite[Theorem~2]{Meyer}; the sparsity conclusion still rests on the geometric mechanism of~\cite[Lemmas~6--7]{Meyer}.
\begin{corollary}\label{cor:firstMeyerCor}
Let \(\lambda\in\mathbb D\) and let $\alpha>0$ be irrational, let \(f_0\in K_{zb_\lambda}\) and for \(n\in\mathbb Z\) define $f_n=b_\lambda^n f_0$ via boundary values on \(\mathbb T\). Then all conclusions of Theorem~\ref{thm:meyer} hold for the measure $\widehat\kappa=\sum_{k,n\in\mathbb Z} \widehat{f_n}(k)\,\delta_{k+\alpha n}$, $k,n\in\mathbb{Z}$.
\end{corollary}
\begin{proof}
By Lemma~\ref{lem:model-space-characterization}, $f_0\in K_{zb_\lambda}$ implies $c(k,n)=0$ whenever $kn<0$. Moreover, since \(|b_\lambda|=1\) on \(\mathbb T\), for
all \(n\in\mathbb Z\),
\begin{align}\label{eq:BlaschkeBoundedness}
        |c(k,n)|
        =
        |\widehat{b_\lambda^n f_0}(k)|
        \le
        \|b_\lambda^n f_0\|_{L^2(\mathbb T)}
        =
        \|f_0\|_{L^2(\mathbb T)}<\infty.
\end{align}
Thus the assumptions of Theorem~\ref{thm:meyer} are met.
\end{proof}

\subsection{Generalized Meyer-type recursion schemes}\label{sec:gen-meyer-recursion}
In this section we examine Meyer's recursion mechanism at the level of general inner
functions. Given an inner function \(\theta\), we consider boundary-value
recursions of the form $f_n=\theta^n f_0$, $n\in\mathbb Z$,
and study which assumptions on the initial conditions \(f_0\) ensure that the
projected Fourier series defines a Radon measure, a tempered distribution, or
a measure with locally finite support. We begin by extending Lemma~\ref{lem:model-space-characterization}. We deliberately extend the set of values of $\alpha$ in our analysis.
\begin{thm}\label{thm:inner-model-space-tempered}
Let \(\theta\) be an inner function, let \(f_0\in K_{z\theta}=H^2(\mathbb D)\ominus z\theta H^2(\mathbb D)\), and let \(\alpha>0\). For \(n\in\mathbb Z\), define $
f_n=\theta^n f_0$
via boundary values on \(\mathbb T\), and set
\[
c(k,n)=\widehat{f_n}(k)=\widehat{\theta^n f_0}(k),\qquad k,n\in\mathbb Z.
\]
Then the formal series
\[
\widehat\kappa_{\theta}
=
\sum_{k,n\in\mathbb Z} c(k,n)\,\delta_{k+\alpha n}
\]
defines a purely atomic Radon measure with locally finite support and at most quadratic growth, hence \(\widehat\kappa_{\theta}\) defines a tempered distribution.
\end{thm}

\begin{proof}
As in Lemma~\ref{lem:model-space-characterization}, \(f_0\in K_{z\theta}\) implies $\widehat{\overline{\theta}\,f_0}(k)=0$, for all $k>0$. Multiplying by further powers of \(\overline{\theta}\), it follows that for every \(m\geq 1\), $\widehat{\overline{\theta}^{\,m}f_0}(k)=\widehat{{\theta}^{-m}f_0}(k)=0$, $k>0$. On the other hand, for \(n\geq 0\) one has \(f_n=\theta^n f_0\in H^2(\mathbb D)\), so $\widehat{f_n}(k)=0$ for $k<0$. Hence \(c(k,n)=0\) whenever \(kn<0\) and therefore $|k+\alpha n|=|k|+\alpha|n|$. Thus, only finitely many pairs \((k,n)\in\mathbb Z^2\) can contribute to an interval \([-R,R]\). Since \(|\theta|=1\) almost everywhere on \(\mathbb T\) for every \(n\in\mathbb Z\) we have
$|c(k,n)|
\le \|f_n\|_{L^2(\mathbb T)}
=
\|f_0\|_{H^2}$.
It follows that
\[
|\widehat\kappa_{\theta}|([-R,R])
\leq
\sum_{|k+\alpha n|\le R} |c(k,n)|
\le
\|f_0\|_{H^2}\,
\#\Bigl\{(k,n)\in\mathbb Z^2:\ |k|+\alpha|n|\le R\Bigr\}.
\]
If \(\alpha\) is rational, several pairs \((k,n)\) may project to the same
point \(k+\alpha n\), but only
finitely many such contributions occur in every compact interval. Thus $\widehat\kappa_{\theta}$ defines a purely atomic Radon measure.
The growth of the right hand side is $\mathcal O(R^2)$ and therefore {$\hat \kappa_{\theta}$} defines a tempered distribution.
\end{proof}

For each \(\zeta\in\mathbb T\), Aleksandrov--Clark theory applied to the
inner function \(z\theta\) gives a finite positive measure
\(\sigma^{z\theta}_\zeta\) on \(\mathbb T\) and a unitary Clark parametrization
\[
C^{z\theta}_\zeta:L^2(\sigma^{z\theta}_\zeta)\longrightarrow K_{z\theta};
\]
see \cite{Clark1972,Saksman2007}. We use this to parametrize admissible
initial conditions.

\begin{corollary}[Clark parametrization of Meyer-type measures]\label{cor:clark}
Let $\theta$ be inner, let $\alpha>0$, fix $\zeta\in\mathbb T$, and let
$\sigma_\zeta^{z\theta}$ be the Aleksandrov-Clark measure of $z\theta$.
For $g_0\in L^2(\sigma_\zeta^{z\theta})$, set $f_0=\mathcal C_\zeta^{z\theta}g_0\in K_{z\theta}$ and define
\[
\widehat\kappa_{\theta,\alpha,\zeta,g_0}
=
\sum_{k,n\in\mathbb Z}
\widehat{\theta^n f_0}(k)\delta_{k+\alpha n}.
\]
Then $\widehat\kappa_{\theta,\alpha,\zeta,g_0}$ is a purely atomic Radon
measure with locally finite support and at most quadratic growth. In particular,
it defines a tempered distribution. 

If moreover \(\theta=B\) is a finite Blaschke product then the Clark
parametrization of \(K_{zB}\) is explicit.
In this case
\[
\sigma^{zB}_\zeta
=
\sum_{\xi B(\xi)=\zeta}
\frac{1}{|(zB)'(\xi)|}\,\delta_\xi,
\]
and every admissible initial condition \(f_0\in K_{zB}\) obtained from this
parametrization has the finite form
\[
f_0(w)
=
\bigl(1-\overline{\zeta}\,wB(w)\bigr)
\sum_{\xi B(\xi)=\zeta}
\frac{g_0(\xi)}{1-\overline{\xi}w}\,
\frac{1}{|(zB)'(\xi)|}.
\]

Moreover, if $\alpha\notin\mathbb Q$, then the mapping is injective:
\[
g_0\in L^2(\sigma_\zeta^{z\theta})
\longmapsto
\widehat\kappa_{\theta,\alpha,\zeta,g_0}.
\]

\end{corollary}
If $\alpha\in\mathbb Q$, the Clark parametrization of the admissible
initial conditions remains valid, but injectivity of the projected measure is not
automatic, since distinct lattice points $(k,n)$ may project to the same
point $k+\alpha n\in\mathbb R$.
\begin{proof}[Proof of Corollary~\ref{cor:clark}]
The Clark transform
\[
\mathcal C_\zeta^{z\theta}:L^2(\sigma_\zeta^{z\theta})\to K_{z\theta}
\]
is unitary. Hence $f_0=\mathcal C_\zeta^{z\theta}g_0$ is an admissible
initial condition in Theorem~\ref{thm:inner-model-space-tempered}, which gives the
Radon property, local finiteness, quadratic growth, and temperedness.

When \(\theta = B\) is a finite Blaschke product, the displayed finite formula
follows from the standard Clark measure formula applied to the inner function
\(zB\), together with the finite atomic form of the Clark parametrization.

If $\alpha\notin\mathbb Q$, then $(k,n)\mapsto k+\alpha n$ is injective
on $\mathbb Z^2$. Thus the projected measure determines every coefficient
$\widehat{\theta^n f_0}(k)$ separately. In particular, the row $n=0$
recovers the Fourier coefficients of $f_0$, and therefore recovers $f_0$.
Since $\mathcal C_\zeta^{z\theta}$ is unitary, $g_0$ is then uniquely
determined. 
\end{proof}

Next we study exponentially decaying arrays. We provide conditions when such arrays yield a Radon measure and a tempered distribution.
\begin{lemma}\label{lem:radon-tempered-projected-tail}
Let \(\alpha\in\mathbb R\), and let \(c=(c(k,n))_{k,n\in\mathbb Z}\) be a
coefficient array. Assume that:

\begin{enumerate}
\item There exist \(A>0\) and \(M\ge0\) such that
\[
        |c(k,n)|
        \le A(1+|k|+|n|)^M ,
        \qquad k,n\in\mathbb Z .
\]

\item There exist \(\varepsilon,C,\rho>0\) such that if $\left|\frac{k}{n}+\alpha\right|<\varepsilon$ and $n\neq0$ then
\[
        |c(k,n)|
        \le C e^{-\rho |n|}
\]
\end{enumerate}
Then the formal series
\[
\widehat\kappa
=
\sum_{k,n\in\mathbb Z} c(k,n)\,\delta_{k+\alpha n}
\]
defines a purely atomic Radon measure of at most polynomial growth
\[
        |\widehat\kappa|([-R,R])=\mathcal{O}(R^{M+2})\ \text{as}\ R\rightarrow\infty,
\]
and hence a tempered distribution. If $\alpha=p/q$ is rational, then $\hat \kappa$ has uniformly discrete support $\text{supp}(\hat\kappa)\subset\frac{1}{q}\mathbb Z$. 
\end{lemma}
\begin{proof}
We first prove that the series defines a purely atomic Radon measure by showing that \(\sum_{k+\alpha n\in I}|c(k,n)|<\infty\) for every compact interval \(I\subset\mathbb R\). If \(R_I>0\) is such that \(I\subset[-R_I,R_I]\) 
\[
        \left|\frac{k}{n}+\alpha\right|
        =
        \frac{|k+\alpha n|}{|n|}
        \le
        \frac{R_I}{|n|}.
\]
Hence, for \(|n|>R_I/\varepsilon\), the exponential decay of tails gives $|c(k,n)|\le C e^{-\rho |n|}$ and for fixed \(n\), there are at most {$2R_I+2$} values of \(k\) such that
\(k+\alpha n\in I\). It follows
\begin{align}\label{eq:tailestimate}
        \sum_{\substack{k+\alpha n\in I\\ |n|>R_I/\varepsilon}}
        |c(k,n)|
        \le
        C{(2R_I+2)}\sum_{|n|>R_I/\varepsilon}e^{-\rho |n|}
        <\infty .
\end{align}
The remaining values of \(n\) form a finite set, and for each such \(n\) only
finitely many \(k\) satisfy \(k+\alpha n\in I\), so that $\sum_{\substack{k+\alpha n\in I}}
        |c(k,n)|<\infty$. Consequently the formal series $\widehat\kappa$ defines a purely atomic Radon measure. If \(\alpha=p/q\) then $k+\alpha n=\frac{qk+pn}{q}\in q^{-1}\mathbb Z$, so that \(\operatorname{supp}(\widehat\kappa)\subset q^{-1}\mathbb Z\). To prove polynomial growth we consider the intervals $|n|\le R/\varepsilon$ and $|n|> R/\varepsilon$. In the first range, \( |k+\alpha n|\le R \) implies $|k|\le R+|\alpha| |n|$ so that there are \(\mathcal{O}(R^2)\) admissible pairs \((k,n)\), and for each of them
\[
        |c(k,n)|
        \le
        A(1+|k|+|n|)^M
        \le
        C R^M,
\]
which yields an overall contribution of \(\mathcal{O}(R^{M+2})\). In the second range, $\left|\frac{k}{n}+\alpha\right| <\varepsilon$ and following Equation~\eqref{eq:tailestimate} we find that the contribution is $\mathcal{O}(R)$.
Combining the two estimates yields $|\widehat\kappa|([-R,R])
        \le
        C'(1+R)^{M+2}$.
\end{proof}
We apply this result in the context of inner functions.
\begin{thm}\label{thm:inner-model-annulus}
Let \(\theta\) be an inner function, let \(f_0\) be holomorphic in an annulus $r<|z|<R$ with $r<1<R$, and let \(\alpha>0\). For \(n\in\mathbb Z\), define $
f_n=\theta^n f_0$
via boundary values on \(\mathbb T\), and set
\[
c(k,n)=\widehat{f_n}(k)=\widehat{\theta^n f_0}(k),\qquad k,n\in\mathbb Z.
\]
Then the formal series
\[
\widehat\kappa_{\theta}
=
\sum_{k,n\in\mathbb Z} c(k,n)\,\delta_{k+\alpha n}
\]
defines a purely atomic Radon measure of at most quadratic growth, hence \(\widehat\kappa_{\theta}\) defines a tempered distribution. If $\alpha=p/q$ is rational, then $\hat \kappa_\theta$ has uniformly discrete support $\text{supp}(\hat\kappa_\theta)\subset\frac{1}{q}\mathbb Z$.
\end{thm}
\begin{proof}
We verify the hypotheses of Lemma~\ref{lem:radon-tempered-projected-tail}. \emph{(1)} Since \(|\theta|=1\) a.e.~on \(\mathbb T\), for
all \(n\in\mathbb Z\), it follows as in Equation~\eqref{eq:BlaschkeBoundedness} that $|c(k,n)|\leq\|f_0\|_{L^2(\mathbb T)}$.
Thus the first assumption of
Lemma~\ref{lem:radon-tempered-projected-tail} holds with \(M=0\). \emph{(2)} 
We have that
\[
        c(k,n)
        =
        \widehat{\theta^n f_0}(k)
        =
        \sum_{\ell\in\mathbb Z}
        \widehat{\theta^n}(k-\ell)\widehat{f_0}(\ell).
\]
Because \(f_0\) is holomorphic in an annulus around \(\mathbb T\), there exist constants \(A,\eta>0\)
such that $|\widehat{f_0}(\ell)|
        \le
        A e^{-\eta|\ell|}$, $\ell\in\mathbb Z$. Choose $\varepsilon,\delta$ so that $\abs{\varepsilon+\delta}<\alpha$. Let \(n\neq0\) and suppose that $\left|\frac{k}{n}+\alpha\right|<\varepsilon$. Split the preceding convolution according to
\[
        \frac{|\ell|}{|n|}\le \delta 
        \qquad\text{and}\qquad
        \frac{|\ell|}{|n|}>\delta.
\]
If \(|\ell|/|n|\le \delta \), then
\[
        \frac{k-\ell}{n}
        =
        \frac{k}{n}-\frac{\ell}{n}<-\alpha+\varepsilon+\delta<0
\]
and by this choice $\widehat{\theta^n}(k-\ell)=0$. For the complementary range, the elementary estimate $|\widehat{\theta^n}(k-l)|\le 1$ yields
\[
        \sum_{|\ell|>\delta |n|}
        |\widehat{\theta^n}(k-\ell)|\,|\widehat{f_0}(\ell)|
        \le
        \sum_{|\ell|>\delta |n|} A e^{-\eta|\ell|}
        \le
        C e^{-\eta\delta |n|}.
\]
Combining the two estimates verifies the second assumption of
Lemma~\ref{lem:radon-tempered-projected-tail}.
\end{proof}
The bottom line of this discussion is that assuming that \(f_0\) is holomorphic in an annulus yields exponential tails and hence Radon measures and temperedness, but not necessarily locally finite support. The model-space assumption \(f_0\in K_{z\theta}\) yields locally finite projected support.
\subsection{Crystalline measures generated by powers of a finite Blaschke product}
We now specialize to finite Blaschke products. Whenever $f_0$ is holomorphic in an annulus containing $T$, Meyer's recursion \(f_n=B^n f_0\), \(n\in\mathbb Z\), yields a Fourier pair of purely atomic Radon
measures defining tempered distribution. Since the projected support may not be locally
finite, this places the construction in the broader framework of pure point diffraction, see \cite{BaakeGrimm2012} and \cite{BaakeStrungaruTerauds2020}. In contrast, when $\alpha>0$ and $f_0\in K_{zB}=H^2\ominus zBH^2$, the resulting measure is crystalline. Using Lemma~\ref{lem:radon-tempered-projected-tail},
exponential tail estimates for the Fourier coefficients of $B^n$ allow one to enlarge the range
of admissible projection parameters beyond the assumption $\alpha>0$ in Section~\ref{sec:gen-meyer-recursion}.
Let
$$B(z)=\prod_{i=1}^d b_{\lambda_i}(z),\quad \{\lambda_i\}_{i=1}^d\subset\mathbb D\setminus\{0\},$$
be a Blaschke product of degree $d\geq1$ with $B(0)\neq0$. Then
\[
B(z)=\frac{P(z)}{Q(z)},
\qquad
P(z)=\sum_{j=0}^d p_j z^j,
\qquad
Q(z)=\sum_{j=0}^d q_j z^j,
\qquad q_0=1,
\]
where \(P\) and \(Q\) have no common zero in \(\overline{\mathbb D}\). Let \(f_0\in H^2(\mathbb D)\), define
\[
f_n=B^n f_0,\qquad n\in\mathbb Z,
\]
and set \(c(k,n)=\widehat{f_n}(k)\) for all \(k,n\in\mathbb Z\). Since \(Qf_{n+1}=Pf_n\), comparison of Fourier coefficients gives, for every \(k,n\in\mathbb Z\),
\begin{equation*}
\sum_{j=0}^d q_j\,c(k-j,n+1)-\sum_{j=0}^d p_j\,c(k-j,n)=0.
\end{equation*}
Thus the array \(c\) satisfies a finite convolution relation on \(\mathbb Z^2\), which can be written in the form $(a_B*c)(m,n)=0$, $m,n\in\mathbb Z$ with $
a_B(j,0)=q_j$, $a_B(j,1)=-p_j$, $0\le j\le d$, and \(a_B(i,j)=0\) otherwise. Projecting each $(k,n)\in\mathbb Z^2$ to $k+\alpha n\in\mathbb R$, yields a formal Dirac series $c\mapsto\widehat{\kappa}=\sum_{k,n}c(k,n)\delta_{k+\alpha n}$ and a finite atomic measure\begin{equation*}
\tau_{B,\alpha}
=
\sum_{j=0}^d q_j\,\delta_j-\sum_{j=0}^d p_j\,\delta_{j+\alpha},
\end{equation*}
with inverse Fourier transform 
\begin{equation}\label{eq:phiB}
\phi_{B,\alpha}(x)
=
\cF^{-1}[\tau_{B,\alpha}](x)
=
\sum_{j=0}^d q_j e^{2\pi i jx}
-
e^{2\pi i\alpha x}\sum_{j=0}^d p_j e^{2\pi i jx}.
\end{equation}

The Fourier coefficients $\widehat{B^n}(k)$ decay exponentially in magnitude outside a compact interval. More precisely, by Proposition~\ref{prop:exp-tails-Bn} in Appendix~\ref{app:finite-blaschke-coeff-bounds} there exists a compact interval $I$, such that for every compact interval \(J\subset\mathbb R\setminus I\) there exist
\(C=C_J,\rho=\rho_J>0\) so that
\[
        |\widehat{B^n}(k)|\le C e^{-\rho|n|}
\]
holds whenever \(n\neq0\) and \(k/n\in J\). 

\begin{thm}\label{thm:general-meyer} Let \(B\) be a finite Blaschke product with \(B(0)\neq0\), let {\(\alpha\in\mathbb R\)} and let \(f_0\) be holomorphic in an annulus \(r<|z|<R\) with \(r<1<R\). Suppose that \(-\alpha\) is strictly separated from \(I_B=[m_B,M_B]\), where
\[ m_B= \sum_{j=1}^d\frac{1-|\lambda_j|^2}{(1+|\lambda_j|)^2}, \qquad M_B= \sum_{j=1}^d\frac{1-|\lambda_j|^2}{(1-|\lambda_j|)^2}. \] 
For \(n\in\mathbb Z\) define \(f_n=B^n f_0\) via boundary values on \(\mathbb T\). Writing \(c(k,n)=\widehat{f_n}(k)\), the formal series \[ \widehat\kappa_B=\sum_{k,n\in\mathbb Z} c(k,n)\,\delta_{k+\alpha n} \] defines a purely atomic Radon measure of at most quadratic growth and hence a tempered distribution.

If $\alpha>-m_B$, then the inverse Fourier transform \(\kappa_B\) is a tempered distribution supported on a uniformly discrete set $\Lambda_B=supp(\kappa_B)$ that is contained
in the zero set $Z(\phi_{B,\alpha})$, cf.~Equation~\eqref{eq:phiB}. If \(x_k<x_{k+1}\) are consecutive points of \(Z(\phi_{B,\alpha})\), then \begin{equation*} \frac{1}{\alpha+M_B} \le x_{k+1}-x_k \le \frac{1}{\alpha+m_B}. \end{equation*} In particular \(\Lambda_B\) is \(1/(\alpha+M_B)\)-separated. 

If, in addition, $\widehat{\kappa}_B$ has locally finite support, then $\kappa_B$ is a sparse crystalline measure. This holds, in particular, if one of the following assumptions is satisfied:
\begin{enumerate} \item \(\alpha>0\) and \(f_0\in K_{zB}=H^2\ominus zBH^2\), \item \(\alpha>0\) and \(c(k,n)=0\) whenever \(kn<0\), \item \(\alpha\in\mathbb Q\). \end{enumerate}
\end{thm}
\begin{proof}
We first verify the hypotheses of Lemma~\ref{lem:radon-tempered-projected-tail}. It
then follows that \(\widehat\kappa_B\) defines a purely atomic tempered distribution
and that \(\phi_{B,\alpha}=\cF^{-1}[\tau_{B,\alpha}]\) is given by \eqref{eq:phiB}.

\emph{(1)} Since \(|B|=1\) on \(\mathbb T\), for all \(n\in\mathbb Z\), as in
Equation~\eqref{eq:BlaschkeBoundedness},
\[
        |c(k,n)|\le \|f_0\|_{L^2(\mathbb T)}.
\]
Thus the first assumption of Lemma~\ref{lem:radon-tempered-projected-tail} holds
with \(M=0\).

\emph{(2)} We have
\[
        c(k,n)=\widehat{B^n f_0}(k)
        =\sum_{\ell\in\mathbb Z}\widehat{B^n}(k-\ell)\widehat{f_0}(\ell).
\]
Because \(f_0\) is holomorphic in an annulus around \(\mathbb T\), there exist
constants \(A,\eta>0\) such that
\(|\widehat{f_0}(\ell)|\le A e^{-\eta|\ell|}\), \(\ell\in\mathbb Z\). By assumption there exist \(\varepsilon,\delta>0\) such that the closed
interval
\[
        J_{\varepsilon,\delta}
        =[-\alpha-\varepsilon-\delta,-\alpha+\varepsilon+\delta]
\]
is disjoint from \(I_B\), i.e.~$J_{\varepsilon,\delta}\subset\mathbb R\setminus [m_B,M_B]$.
Let \(n\neq0\) and suppose that \(\left|\frac{k}{n}+\alpha\right|<\varepsilon\).
Split the preceding convolution according to
\[
        \frac{|\ell|}{|n|}\le \delta
        \qquad\text{and}\qquad
        \frac{|\ell|}{|n|}>\delta.
\]
If \(|\ell|/|n|\le\delta\), then
\[
        \frac{k-\ell}{n}=\frac{k}{n}-\frac{\ell}{n}
        \in J_{\varepsilon,\delta}
\]
so that
\[
        |\widehat{B^n}(k-\ell)|\le C e^{-\rho|n|}
\]
Consequently,
\[
        \sum_{|\ell|\le\delta |n|}
        |\widehat{B^n}(k-\ell)|\,|\widehat{f_0}(\ell)|
        \le Ce^{-\rho|n|}\sum_{\ell\in\mathbb Z}|\widehat{f_0}(\ell)|
        \le C_2e^{-\rho_1|n|}.
\]
For the complementary range, the elementary estimate \(|\widehat{B^n}(k-\ell)|\le1\)
gives
\[
        \sum_{|\ell|>\delta |n|}
        |\widehat{B^n}(k-\ell)|\,|\widehat{f_0}(\ell)|
        \le \sum_{|\ell|>\delta |n|}Ae^{-\eta|\ell|}
        \le C_3e^{-\eta\delta |n|}.
\]
Combining the two estimates verifies the second assumption of
Lemma~\ref{lem:radon-tempered-projected-tail}.

Assume now that \(\alpha>-m_B\). We prove the stated bounds for the distance between two consecutive zeros of \(\phi_{B,\alpha}\).
Represent \(B\) on \(\mathbb T\) as
\[
        B(e^{i\theta})=e^{i\psi_B(\theta)},
        \qquad
        \psi_B(\theta+2\pi)=\psi_B(\theta)+2\pi d.
\]
Since \(B\) is finite, \(\psi_B\) is smooth and
\[
        \psi_B'(\theta)=
        \sum_{j=1}^d\frac{1-|\lambda_j|^2}{|e^{i\theta}-\lambda_j|^2},
        \qquad \theta\in\mathbb R,
\]
where the zeros are counted with multiplicity. By the definitions of \(m_B\) and
\(M_B\), this gives \(m_B\le \psi_B'(\theta)\le M_B\). Hence
\[
        F(\theta)=\psi_B(\theta)+\alpha\theta
\]
is strictly increasing, with \(m_B+\alpha\le F'(\theta)\le M_B+\alpha\).
Since \(Q(e^{2\pi ix})\neq0\) for \(x\in\mathbb R\), \(\phi_{B,\alpha}(x)=0\) is
equivalent to
\[
        B(e^{2\pi ix})=e^{-2\pi i\alpha x},
\]
and therefore
\[
        Z(\phi_{B,\alpha})=\{x\in\mathbb R:F(2\pi x)\in2\pi\mathbb Z\}.
\]
Since \(F\) is strictly increasing, for each \(k\in\mathbb Z\) there is a unique
solution \(x_k\) of \(F(2\pi x_k)=2\pi k\), and the zero set is locally finite. Computing derivatives and plugging in $x_k$ one finds
\[
        \phi_{B,\alpha}'(x_k)
        =-2\pi i Q(e^{2\pi ix_k})(\psi_B'(2\pi x_k)+\alpha),
\]
which has strictly positive modulus. As before, this implies that $\kappa_B$ contains no derivative of Dirac masses so that \(\kappa_B\) is a purely
atomic Radon measure. If \(x_k<x_{k+1}\) are consecutive zeros, then the mean value
theorem gives
\[
        2\pi=F' (\xi_k)(2\pi x_{k+1}-2\pi x_k)
\]
for some \(\xi_k\in(2\pi x_k,2\pi x_{k+1})\). The bounds on \(F'\) give the stated
separation estimates. If \(f_0\in K_{zB}\), then, as in the proof of
Theorem~\ref{thm:inner-model-space-tempered}, \(c(k,n)=0\) whenever \(kn<0\) so that
\(\{k+\alpha n:c(k,n)\neq0\}\) is locally finite. Then \(\kappa_B\) is a sparse
crystalline measure.
\end{proof}
{
By Lemma~\ref{lem:radon-tempered-projected-tail} with \(M=0\), as used in the proof of Theorem~\ref{thm:general-meyer}, one has
\[
        |\widehat\kappa_B|([-R,R])=\cO(R^2).
\]
To sharpen this quadratic growth bound, we first record two standard row estimates that will also be used in Section~\ref{sec:fourier-trunc-sampling}. A classical estimate of Blyudze--Shimorin~\cite{BlyudzeShimorin,SzehrZaroufJAM}, using for negative powers that $B^{-m}=\overline{B^m}$ on $\mathbb T$, gives a constant $C_B>0$ such that
\begin{align}\label{eq:Blyudze}
\sum_{q\in\mathbb Z}
|\widehat{B^n}(q)|
\le
C_B(1+|n|)^{1/2},
\qquad n\in\mathbb Z.
\end{align}
Moreover, since $f_0$ is holomorphic in an annulus around $\mathbb T$, its Fourier coefficients decay exponentially and are therefore summable. Young's inequality then yields a constant $C_{B,f_0}>0$ such that
\begin{equation}\label{eq:absoluteSum}
        \sum_{k\in\mathbb Z}|c(k,n)|\leq C_{B,f_0}(1+|n|)^{1/2},
        \qquad n\in\mathbb Z.
\end{equation}

\begin{corollary}\label{cor:kappahat-three-halves-growth}
Let \(B\) be a finite Blaschke product with \(B(0)\neq0\), let \(f_0\) be holomorphic in an annulus \(r_0<|z|<R_0\), where \(r_0<1<R_0\), and let \(\alpha\in\mathbb R\). Assume that
\[
        -\alpha\notin[m_B,M_B].
\]
Then
\[
        |\widehat\kappa_B|([-R,R])=\cO(R^{3/2}),
        \qquad R\to\infty.
\]
\end{corollary}

\begin{proof}
By \eqref{eq:absoluteSum},
\[
        \sum_{k\in\mathbb Z}|c(k,n)|
        \le C_{B,f_0}(1+|n|)^{1/2},
        \qquad n\in\mathbb Z.
\]
We now prove the exponential estimate used below for indices \(k,n\) such that
\(|k/n+\alpha|\) is small. Since \(-\alpha\notin[m_B,M_B]\), choose
\(\varepsilon>0\) such that
\[
        [-\alpha-2\varepsilon,-\alpha+2\varepsilon]
        \subset \mathbb R\setminus[m_B,M_B].
\]
By Proposition~\ref{prop:exp-tails-Bn} in
Appendix~\ref{app:finite-blaschke-coeff-bounds}, the multiplicative constant is
\(1\) for the coefficients of \(B^n\). Combining this estimate with the exponential
decay of the Fourier coefficients of \(f_0\) gives constants \(C,\rho>0\) such that
\[
        |c(k,n)|\le C e^{-\rho|n|}
\]
whenever \(n\neq0\) and \(|k/n+\alpha|<\varepsilon\). To see this, write
\[
        c(k,n)=\sum_{\ell\in\mathbb Z}\widehat f_0(\ell)\widehat{B^n}(k-\ell).
\]
The terms with \(|\ell|\le\varepsilon |n|\) are controlled by
Proposition~\ref{prop:exp-tails-Bn}, while the remaining terms are exponentially
small because \(\widehat f_0\) decays exponentially and \(|\widehat{B^n}(q)|\le1\).

We now split the mass of \(\widehat\kappa_B\) on \([-R,R]\) according to
\(|n|\le R/\varepsilon\) and \(|n|>R/\varepsilon\). In the first range,
\[
\sum_{|n|\le R/\varepsilon}\sum_{k\in\mathbb Z}|c(k,n)|
\le C\sum_{|n|\le R/\varepsilon}(1+|n|)^{1/2}
=\cO(R^{3/2}).
\]
In the second range, the condition \(|k+\alpha n|\le R\) implies
\[
        \left|\frac{k}{n}+\alpha\right|\le \frac{R}{|n|}<\varepsilon.
\]
Hence
\[
\sum_{\substack{|n|>R/\varepsilon\\ |k+\alpha n|\le R}} |c(k,n)|
\le C(2R+2)\sum_{|n|>R/\varepsilon} e^{-\rho|n|}
=\cO(1).
\]
Combining the two estimates gives the result.
\end{proof}

The exponent $3/2$ is sharp in general. Indeed, let $\lambda\in(0,1)$ \(B=b_\lambda\), $f_0=1$, and let \(\alpha>0\) be irrational. By Corollary~\ref{cor:l1-short} in Appendix~\ref{app:meyer-bulk-lower}, there exist \(0<a_1<a_2\), \(c>0\), and \(n_0\ge1\) such that, for every \(n\ge n_0\),
\[
        \sum_{a_1n\le k\le a_2n}
        \bigl|\widehat{b_\lambda^n}(k)\bigr|
        \ge c\sqrt n.
\]
Since irrationality of \(\alpha\) makes the projected points \(k+\alpha n\) distinct, there exists a constant \(c_1>0\) such that, for all sufficiently large \(R\),
\[
        |\widehat\kappa_B|([0,R])
        \ge c\sum_{n_0\le n\le R/(a_2+\alpha)}\sqrt n
        \ge c_1 R^{3/2}.
\]
Thus Corollary~\ref{cor:l1-short} provides the matching lower bound and proves that the exponent $3/2$ in Corollary~\ref{cor:kappahat-three-halves-growth} cannot be improved in general.
}
%
%
%

We illustrate the theorem with examples, showing how some of Meyer's constructions appear in our setting and providing quantitative sparsity estimates. Notice that the sparsity estimates of Theorem~\ref{thm:general-meyer} depend on the recursion scheme but not on the initial conditions. For fixed degree \(d\), choosing zeros deeper in the disk (makes \(\psi_B'\) flatter and) decreases \(M_B\) toward its minimal value \(d\), thereby improving the lower bound on separation.



%
%
\begin{example}[Meyer's examples \(\gamma_0,\gamma_1,\gamma_2\)]\label{exp:meyersGammas}
We recapitulate the crystalline measure constructions \(\gamma_0,\gamma_1,\gamma_2\) introduced by Meyer in \cite[Definitions~2-4]{Meyer}. All three arise from the recurrence \(f_n=b_{1/2}^nf_0\) with \(\alpha=\sqrt{2}\) but are characterized by different initial conditions. These are: \emph{1)} \(\gamma_0\) is characterized by \(c(0,0)=1\) and \(c(m,0)=0\) for \(m\neq0\) (corresponding to \(f_0=1\)), \emph{2)} \(\gamma_1\) is characterized by \(c(m,0)=(1/2)^m\) for \(m\ge0\) and \(c(m,0)=0\) for \(m\le-1\) (corresponding to \(f_0=\frac{1}{1-z/2}\)) and \emph{3)} \(\gamma_2\) is characterized by \(c(m,0)=(1/2)^{|m|}\) for \(m\le0\) and \(c(m,0)=0\) for \(m\ge1\) (corresponding to \(f_0=(1-z^{-1}/2)^{-1}\)).

For $\gamma_0$ and $\gamma_1$ the corresponding initial conditions lie in $K_{zb_1/2}$, so Corollary \ref{cor:firstMeyerCor} yields a sparse crystalline measure. On the other hand, the initial conditions for $\gamma_2$ are not in $H^2$ but they are holomorphic in an annulus around $\mathbb T$, so that Theorem~\ref{thm:general-meyer} applies and shows that $\kappa_B$ is a tempered distribution. Moreover, $\kappa_B$ is a sparse crystalline measure by Theorem~\ref{thm:general-meyer} point \emph{(2)} using that with \(\lambda=1/2\), one has \(b_\lambda^n f_0\in zH^2\) for \(n\ge1\) and $ b_\lambda^{-m}f_0$ has only non-positive Fourier frequencies for \(m\ge1\) (rather than exponential decay). Computation gives \(M_{b_{1/2}}=3\) and the separation constant\footnote{Meyer works with the angular variable \(t\in\mathbb R/2\pi\mathbb Z\), which results in an angular separation of \(2\pi\delta_{\mathrm{Meyer}}\)} $\delta_{{Meyer}}
        =
        \frac{1}{\sqrt2+3}
        \approx 0.2265$.
\end{example}
For a finite Blaschke product of degree \(d\), one always has
\[
\sup_{\theta}\psi_B'(\theta)\ge \frac{1}{2\pi}\int_0^{2\pi}\psi_B'(\theta)\,\diff\theta=d.
\]
In particular for fixed \(\alpha\) the separation bound delivered by
Theorem~\ref{thm:general-meyer} in degree \(d\) is at most \(1/(\alpha+d)\).
Hence for \(\alpha=\sqrt2\) no product of degree \(d\ge3\) can improve on the Meyer
benchmark \(1/(\sqrt2+3)\) at the level of this bound.
\begin{example}[A symmetric degree-\(d\) family]
Extending Meyer's recursion scheme for degree-1 Blaschke products, consider for $\lambda\in(0,1)$ the family
\[
        B_{d,\lambda}(z)
        =
        \frac{z^d-\lambda^d}{1-\lambda^d z^d}
        =
        \prod_{j=0}^{d-1} b_{\lambda e^{2\pi i j/d}}(z).
\]
The interest of this family is that extrema of the phase derivative can be computed explicitly,
\[
        \inf_\theta \psi'_{B_{d,\lambda}}(\theta)
        =
        d\,\frac{1-\lambda^d}{1+\lambda^d},
        \qquad
        \sup_\theta \psi'_{B_{d,\lambda}}(\theta)
        =
        d\,\frac{1+\lambda^d}{1-\lambda^d}.
\]
This yields
\[
        \frac{1}{\alpha+d\frac{1+\lambda^d}{1-\lambda^d}}
        \le
        x_{k+1}-x_k
        \le
        \frac{1}{\alpha+d\frac{1-\lambda^d}{1+\lambda^d}}.
\]
For fixed $d$ the lower bound $\bigl(\alpha+d(1+\lambda^d)/(1-\lambda^d)\bigr)^{-1}$ is decreasing in
$\lambda$ on $(0,1)$; its supremum is $1/(\alpha+d)$, approached as $\lambda\to0^+$.
Tuning $\lambda$ one thus can increase the separation, where for degree one and \(\alpha=\sqrt2\), one finds $\sup_{0<\lambda<1}\delta_{1,\lambda,\sqrt2}=\frac{1}{\sqrt2+1}>\delta_{Meyer}$ in Example~\ref{exp:meyersGammas}.

\end{example}

We conclude this section with a brief discussion of the decay of the coefficients $\widehat{b_\lambda^n}(k)$ in the so-called {central region} and point out a small correction to the discussion in \cite[Section~2, Lemma~2]{Meyer} - in fact the asymptotic behavior of $\widehat{b_\lambda^n}(k)$ over the entire range of $k/n$ is known, see e.g.~\cite{SzehrZaroufJAM,SzehrZaroufJAT2022,BFZCA,BFZIMRN}.

A standard approach to determine the asymptotics is to express the coefficients $\widehat{b_\lambda^n}(k)$ in terms of the Jacobi polynomials
with first varying parameter by \cite[Formula~(1.1)]{SzehrZaroufJAT2022} and then apply the asymptotic expansion formulas for Jacobi polynomials.

More precisely the central region is characterized by the inequalities
\[
        \alpha_0 n+\omega(n^{1/3})\le k\le
        \alpha_0^{-1}n-\omega(n^{1/3}),
        \qquad
        \alpha_0=\frac{1-\lambda}{1+\lambda},
\]
where $\omega(n^{1/3})/n^{1/3}\to+\infty$ and $\omega(n^{1/3})=o(n)$. This yields~\cite{SzehrZaroufJAM,SzehrZaroufJAT2022}
\[
        \widehat{b_\lambda^n}(k)
        =
        \sqrt{\frac{2}{n\pi}}
        \frac{\cos(nh(\varphi_+)-\pi/4)+\varepsilon_n(k)}
        {\sqrt{k/n}\,\bigl[(\alpha_0^{-1}-k/n)(k/n-\alpha_0)\bigr]^{1/4}},
\]
with $\varepsilon_n(k)\to0$ as $n\to\infty$ in this range and where representations of $\varphi_+$ and $h\neq0$ are provided in Appendix~\ref{app:asymptotic-formulas}. Notice that the leading term is
oscillatory in $n$, and there is no lower bound of order $n^{-1/2}$ for every index $k$ in the central range. 

To recover the conclusion of \cite{Meyer} that the corresponding Fourier-side measure is
not translation bounded one may follow the following line: By Lemma~\ref{lem:osc-robust-lower} in Appendix~\ref{app:meyer-bulk-lower} there are $0<a_1<a_2$, $\eta>0$, $c>0$, and sets
$E_n\subset [a_1n,a_2n]\cap\mathbb Z$, with $\#E_n\ge \eta n$, such that
\[
        |\widehat{b_\lambda^n}(k)|\ge cn^{-1/2}
\]
for all large $n$ and all $k\in E_n$. Since $\alpha$ is irrational in~\cite{Meyer} the points
$k+\alpha n$ are distinct. Therefore, if $1\le n\le R/(a_2+\alpha)$ and
$k\in E_n$, then $0\le k+\alpha n\le R$, and hence
\[
        |\widehat\kappa_0|([0,R])
        \ge
        \sum_{1\le n\le R/(a_2+\alpha)}\sum_{k\in E_n}
        |\widehat{b_\lambda^n}(k)|
        \ge
        c\sum_{1\le n\le R/(a_2+\alpha)} n^{1/2}.
\]
Since $[0,R]$ is covered by at most $ R+1$ intervals of length
one, we have $|\widehat\kappa_0|([0,R])
        \le
        \sum_{j=0}^{\lfloor R\rfloor}
        |\widehat\kappa_0|([j,j+1])$. Therefore there exists $j_R\in\mathbb N$ such that
\[
        |\widehat\kappa_0|([j_R,j_R+1])\geq K R^{1/2},\qquad K>0.
\]

\section{Truncation and sampling for Meyer--Blaschke measures}\label{sec:fourier-trunc-sampling}
We retain the notation introduced above:
\[
\widehat\kappa_B=
\sum_{k,n\in\mathbb Z} c(k,n)\delta_{k+\alpha n},
\]
where \(B\) is a finite Blaschke product, \(f_0\) is the auxiliary annulus function, \(\alpha>0\), and \(c(k,n)\) denotes the coefficient in the \(n\)-th Blaschke row.

We consider a finite truncation in two distinct variables: the row index \(n\) and the projected spectral variable \(k+\alpha n\), keeping only the terms with \(|n|\le N\) and \(|k+\alpha n|\le K\). The estimates below control the two resulting errors separately against Schwartz test functions.

The same truncation is then inserted into a shifted Poisson identity. Under a spectral gap assumption, the identity gives an exact sampling formula for band-limited functions. For general Schwartz test functions, the convergence estimates for the Meyer-Blaschke measures show that, provided \(K=K(N)\) grows sufficiently fast, the finite Fourier-side truncations converge to \(\widehat\kappa_B\) in \(\mathcal S'(\mathbb R)\), with an explicit error bound.

The present discussion connects with the classical theory of truncated sampling expansions, notably D.~Jagerman~\cite{Jagerman1966} and A.~J~.Jerri~\cite[Sec.~VI-A]{Jerri1977}.

\subsection{Finite truncation of \texorpdfstring{\(\widehat\kappa_B\)}{kappa-hat}}
\label{sec:finite-truncation}

For integers $N\ge1$ and $K\ge0$, consider the truncated series
\[
\widehat\kappa_{B;N,K}
=
\sum_{\substack{|n|\le N\\ |k+\alpha n|\le K}}
c(k,n)\delta_{k+\alpha n}.
\]
Thus \(\widehat\kappa_{B;N,K}\) is obtained from \(\widehat\kappa_B\) by deleting
two kinds of terms: all rows with \(|n|>N\), and, in the remaining rows, all atoms
whose projected frequency lies outside \([-K,K]\). Our goal in this section is to provide conditions such that the
two discarded tails are small when tested against Schwartz functions.

Throughout this subsection, let $B$ be a finite Blaschke product, let $f_0$ be holomorphic in an annulus $r<|z|<R$, $r<1<R$, and let $\alpha>0$. The stronger hypotheses ensuring that the inverse Fourier transform $\kappa_B$ is a sparse crystalline measure are not needed until the sampling interpretation in Subsection~\ref{sec:sampling-reconstruction}.

We shall use the row estimates~\eqref{eq:Blyudze} and~\eqref{eq:absoluteSum} established above, together with the following elementary weighted convolution estimate.
Using the inequality
\[
(1+|y+\ell|)^{-M}
\le
(1+|\ell|)^M(1+|y|)^{-M},
\]
one verifies that there exist constants $C_{M,f_0}>0$ such that for every $M>0$, 
\begin{equation}
\sum_{\ell\in\mathbb Z}
|\widehat{f_0}(\ell)|(1+|y+\ell|)^{-M}
\le
C_{M,f_0}(1+|y|)^{-M},
\qquad y\in\mathbb R.\label{eq:weightedConvolution}
\end{equation}

\begin{lemma}
\label{lem:weighted-summability-kappahatB}
There exists a constant $C_{M,B,f_0,\alpha}>0$ so that
\begin{equation*}
        \sum_{k\in\mathbb Z}
        |c(k,n)|(1+|k+\alpha n|)^{-M}<C_{M,B,f_0,\alpha}(1+|n|)^{1/2-M}.
\end{equation*}
Consequently if $M>3/2$ then for every
\(\varphi\in\mathcal S(\mathbb R)\),
\begin{equation*}
        |\widehat\kappa_B(\varphi)|
        \le
        C_{M,B,f_0,\alpha}\sup_{x\in\mathbb R}(1+|x|)^M|\varphi(x)|.
\end{equation*}
\end{lemma}

\begin{proof}
By the triangle inequality, Tonelli's theorem for non-negative series, and the weighted convolution estimate~\eqref{eq:weightedConvolution} with $y=q+\alpha n$ we obtain
\begin{align*}
\sum_{k\in\mathbb Z}|c(k,n)|(1+|k+\alpha n|)^{-M}
&\le
\sum_{q\in\mathbb Z}|\widehat{B^n}(q)|
\sum_{\ell\in\mathbb Z}
|\widehat{f_0}(\ell)|(1+|q+\ell+\alpha n|)^{-M}\\
&\leq C_{M,f_0}
\sum_{q\in\mathbb Z}
|\widehat{B^n}(q)|(1+|q+\alpha n|)^{-M}.
\end{align*}

Now the assumption $\alpha>0$ gives a uniform lower bound on the projected variable. If $n>0$, then $B^n\in H^2$, so $\widehat{B^n}(q)=0$ for $q<0$, and hence
\[
|q+\alpha n|\ge \alpha |n|
\qquad
\text{on the support of } \widehat{B^n}.
\]
If $n<0$, then $B^n=\overline{B^{|n|}}$ on $\mathbb T$, so $\widehat{B^n}(q)=0$ for $q>0$, and again
\[
|q+\alpha n|\ge \alpha |n|
\qquad
\text{on the support of } \widehat{B^n}.
\]
The case $n=0$ is immediate. Therefore
\[
(1+|q+\alpha n|)^{-M}
\le
C_{M,\alpha}(1+|n|)^{-M}
\]
on the support of $\widehat{B^n}$. This yields, together with the Blyudze-Shimorin estimate~\eqref{eq:Blyudze}, that
\begin{align*}
\sum_{k\in\mathbb Z}|c(k,n)|(1+|k+\alpha n|)^{-M}
&\le
C_{M,B,f_0,\alpha}
(1+|n|)^{-M}
\sum_{q\in\mathbb Z}|\widehat{B^n}(q)|  \\
&\le
C_{M,B,f_0,\alpha}
(1+|n|)^{1/2-M}.
\end{align*}
If \(M>3/2\), then $\sum_{n\in\mathbb Z}(1+|n|)^{1/2-M}<\infty$
and hence the weighted double series is finite. Finally, for $\varphi\in\mathcal{S}(\mathbb R)$,
\[
|\varphi(k+\alpha n)|
\le
\sup_{x\in\mathbb R}(1+|x|)^M|\varphi(x)|
(1+|k+\alpha n|)^{-M}.
\]
Multiplying by $|c(k,n)|$, summing over $(k,n)$ yields the claimed bound on $\widehat{\kappa}_B(\varphi)$.
\end{proof}

We now quantify how $\widehat\kappa_{B;N,K}$ approximates $\widehat{\kappa}_B$.
\begin{thm}\label{thm:truncation-kappaB}
Let \(M>3/2\) and set
\[
        \mathcal N_M(\varphi)=\sup_{x\in\mathbb R}(1+|x|)^M|\varphi(x)|.
\]
There exists a constant \(C_{M,B,f_0,\alpha}>0\) such that, for all $N\geq1$ and $K\geq0$ for all
\(\varphi\in\mathcal S(\mathbb R)\),
\begin{equation*}
        |(\widehat\kappa_B-\widehat\kappa_{B;N,K})(\varphi)|
        \le
        C_{M,B,f_0,\alpha}\,\mathcal N_M(\varphi)
        \left(N^{3/2-M}+N^{3/2}(1+K)^{-M}\right).
\end{equation*}
Consequently, if \(K=K(N)\to\infty\) and
\(N^{3/2}(1+K(N))^{-M}\to0\), then
\begin{align*}
\widehat\kappa_{B;N,K(N)}\to\widehat\kappa_B\quad \text{in}\quad \mathcal S'(\mathbb R).
\end{align*}
\end{thm}

\begin{proof}
The omitted terms are those with either $|n|>N$, or with $|n|\le N$ and $|k+\alpha n|>K$. Hence
\begin{align*}
|(\widehat\kappa_B-\widehat\kappa_{B;N,K})(\varphi)|
&\le
\sum_{\substack{|n|>N\\ k\in\mathbb Z}}
|c(k,n)||\varphi(k+\alpha n)|
+
\sum_{\substack{|n|\le N\\ |k+\alpha n|>K}}
|c(k,n)||\varphi(k+\alpha n)|.
\end{align*}
Using Lemma~\ref{lem:weighted-summability-kappahatB} for the first term gives
\begin{align*}
\sum_{\substack{|n|>N\\ k\in\mathbb Z}}
|c(k,n)||\varphi(k+\alpha n)|
&\le
C_{M,B,f_0,\alpha}\mathcal N_M(\varphi)
\sum_{|n|>N}(1+|n|)^{1/2-M}  \\
&\le
C_{M,B,f_0,\alpha}\mathcal N_M(\varphi)N^{3/2-M}.
\end{align*}
For the second term, if $|k+\alpha n|>K$, then
$|\varphi(k+\alpha n)|
\le
\mathcal N_M(\varphi)(1+K)^{-M}$.
This yields
\begin{align*}
\sum_{\substack{|n|\le N\\ |k+\alpha n|>K}}
|c(k,n)||\varphi(k+\alpha n)|
&\le
\mathcal N_M(\varphi)(1+K)^{-M}
\sum_{|n|\le N}\sum_{k\in\mathbb Z}|c(k,n)|  \\
&\le
C_{B,f_0}\mathcal N_M(\varphi)(1+K)^{-M}
\sum_{|n|\le N}(1+|n|)^{1/2}  \\
&\le
C_{B,f_0}\mathcal N_M(\varphi)
N^{3/2}(1+K)^{-M},
\end{align*}
where we made use of the estimate~\eqref{eq:absoluteSum} for the second inequality. Combining the two estimates proves the claimed inequality. The convergence statement follows because $M>3/2$ implies $N^{3/2-M}\to0$, while the second term tends to zero by assumption on $K(N)$.
\end{proof}

\subsection{Fourier-side truncation for Poisson identities and reconstruction}
\label{sec:sampling-reconstruction}

We use the Poisson identity for the Meyer-Blaschke measures in two ways: a spectral gap gives the familiar exact reconstruction formula for band-limited functions, while Theorem~\ref{thm:truncation-kappaB} quantifies finite truncations of the Fourier-side sum for general Schwartz functions.

For a crystalline measure
\[
\kappa=\sum_{\lambda\in\Lambda}a_\lambda\delta_\lambda,
\qquad
\widehat\kappa=\sum_{s\in S}d_s\delta_s,
\]
the generalized Poisson identity applied to the test function $s\mapsto\widehat\varphi(\xi-s)$ gives a shifted Poisson formula
\begin{equation*}
\sum_{\lambda\in\Lambda}
a_\lambda\varphi(\lambda)e^{-2\pi i\xi\lambda}
=
\sum_{s\in S}
d_s\widehat\varphi(\xi-s),
\qquad \xi\in\mathbb R.
\end{equation*}

We first observe an elementary mechanism which turns the summation formula into a sampling formula in the band-limited setting. Fix $B_0>0$, and assume that $0\in S$, $d_0\ne0$
and that the remaining spectral atoms avoid the interval $[-2B_0,2B_0]$:
\begin{equation*}
(S\setminus\{0\})\cap[-2B_0,2B_0]=\emptyset.
\end{equation*}
For the Meyer--Blaschke measures considered here this spectral gap exists, which is a consequence of the local finiteness of the Fourier-side support rather than of the Poisson identity alone.

If $\text{supp}(\widehat\varphi)\subset[-B_0,B_0]$, then for $\xi\in[-B_0,B_0]$ and $s\in S\setminus\{0\}$, one has $\widehat\varphi(\xi-s)=0$. In this situation the Fourier-side sum reduces to its central term,
\[
\sum_{s\in S}
d_s\widehat\varphi(\xi-s)
=
d_0\widehat\varphi(\xi),
\qquad \xi\in[-B_0,B_0].
\]
Using the shifted Poisson formula this gives
\begin{equation*}
d_0\widehat\varphi(\xi)
=
\sum_{\lambda\in\Lambda}
a_\lambda\varphi(\lambda)e^{-2\pi i\xi\lambda},
\qquad \xi\in[-B_0,B_0].
\end{equation*}
Since $\widehat\varphi$ is supported in $[-B_0,B_0]$, inverse Fourier transform gives
\begin{equation*}
\varphi(t)
=
\frac{2B_0}{d_0}
\sum_{\lambda\in\Lambda}
a_\lambda\varphi(\lambda)
\operatorname{sinc}\bigl(2B_0(t-\lambda)\bigr),
\end{equation*}
where $\operatorname{sinc}(x)=\sin(\pi x)/(\pi x)$. The identity is understood in the distributional sense, and pointwise wherever the sampling series converges. Thus, under this gap condition, the samples $\varphi(\lambda)$, $\lambda\in\Lambda$, determine the function $\varphi$.

We now return to the Meyer-type measures associated with finite Blaschke products. The preceding reduction uses the compact support of $\widehat\varphi$ and the spectral gap to collapse the Fourier-side sum to its central atom. For general Schwartz functions the full Fourier-side sum is present and the truncation estimate from Theorem~\ref{thm:truncation-kappaB} becomes relevant.

\begin{corollary}[Finite spectral approximation in the Poisson identity]
\label{cor:finite-spectral-approximation}
Assume the hypotheses of Theorem~\ref{thm:truncation-kappaB}, and assume in addition that the inverse Fourier transform $\kappa_B$ is a crystalline measure. Let $I\subset\mathbb R$ be compact and let $M>3/2$. Then there exists a constant $C_{M,I,B,f_0,\alpha}>0$ such that, for every $\varphi\in\mathcal S(\mathbb R)$, every $N\ge1$, every $K\ge0$, and every $\xi\in I$,
\begin{align*}
\Bigg|
\sum_{k,n\in\mathbb Z}
c(k,n)&\widehat\varphi(\xi-k-\alpha n)-
\sum_{\substack{|n|\le N\\ |k+\alpha n|\le K}}
c(k,n)\widehat\varphi(\xi-k-\alpha n)
\Bigg| \\
&\qquad\le
C_{M,I,B,f_0,\alpha}
\sup_{u\in\mathbb R}
(1+|u|)^M|\widehat\varphi(u)|
\left(
N^{3/2-M}
+
N^{3/2}(1+K)^{-M}
\right).
\end{align*}
Consequently, if $K=K(N)\to\infty$ and $N^{3/2}(1+K(N))^{-M}\to0$,
then the finite spectral sums converge uniformly on $I$ to the Fourier-side sum in the Poisson identity.

If, moreover, $S=\operatorname{supp}(\widehat\kappa_B)$ satisfies the gap condition
\[
0\in S,\qquad d_0=\widehat\kappa_B(\{0\})\ne0,\qquad
(S\setminus\{0\})\cap[-2B_0,2B_0]=\varnothing,
\]
and if $\operatorname{supp}\widehat\varphi\subset[-B_0,B_0]$, then the same finite sums approximate the central term $d_0\widehat\varphi(\xi)$ uniformly for $\xi\in[-B_0,B_0]$:
\[
\begin{aligned}
&\left|
d_0\widehat\varphi(\xi)-
\sum_{\substack{|n|\le N\\ |k+\alpha n|\le K}}
c(k,n)\widehat\varphi(\xi-k-\alpha n)
\right| \\
&\qquad\le
C_{M,B_0,B,f_0,\alpha}
\sup_{u\in\mathbb R}(1+|u|)^M|\widehat\varphi(u)|
\left(
N^{3/2-M}
+
N^{3/2}(1+K)^{-M}
\right),
\qquad \xi\in[-B_0,B_0].
\end{aligned}
\]
\end{corollary}

\begin{proof}
For the Meyer-Blaschke measure, the Fourier-side sum in the shifted Poisson formula is $$\sum_{k,n\in\mathbb Z}
c(k,n)\widehat\varphi(\xi-k-\alpha n).$$ Applying Theorem~\ref{thm:truncation-kappaB} to $\psi_\xi(s)=\widehat\varphi(\xi-s)$ gives
\begin{align*}
\Bigg|
\sum_{k,n\in\mathbb Z}
c(k,n)&\widehat\varphi(\xi-k-\alpha n)-
\sum_{\substack{|n|\le N\\ |k+\alpha n|\le K}}
c(k,n)\widehat\varphi(\xi-k-\alpha n)
\Bigg| \\
&\qquad\le
C_{M,B,f_0,\alpha}
\mathcal N_M(\psi_\xi)
\left(
N^{3/2-M}
+
N^{3/2}(1+K)^{-M}
\right).
\end{align*}
For $\xi\in I$,
\[
\mathcal N_M(\psi_\xi)
=
\sup_{s\in\mathbb R}
(1+|s|)^M|\widehat\varphi(\xi-s)|
\le
C_{M,I}
\sup_{u\in\mathbb R}
(1+|u|)^M|\widehat\varphi(u)|.
\]
This proves the uniform estimate. The convergence statement follows immediately. Under the gap condition, the shifted Poisson identity gives
\[
\sum_{k,n\in\mathbb Z}
c(k,n)\widehat\varphi(\xi-k-\alpha n)
=
d_0\widehat\varphi(\xi),
\qquad \xi\in[-B_0,B_0],
\]
and the second estimate follows from the first with $I=[-B_0,B_0]$.
\end{proof}

If $\alpha=p/q\in\mathbb Q$, with $(p,q)=1$ then the projected spectral set is contained in the lattice $q^{-1}\mathbb Z$. The Fourier-side measure may be regrouped as a weighted lattice comb
\[
\widehat\kappa_B
=
\sum_{m\in\mathbb Z}d_m\delta_{m/q},
\qquad
d_m=
\sum_{\substack{k,n\in\mathbb Z\ qk+pn=m}}c(k,n),
\]
with the sums understood in the distributional sense justified by the preceding estimates. If, in addition, $\kappa_B=\sum_{\lambda\in\Lambda_B}a_\lambda\delta_\lambda$ is crystalline, then the measure is pure point on both sides: the physical support $\Lambda_B$ gives the sampling nodes, while the spectral support is a weighted subcomb of $q^{-1}\mathbb Z$. In this case the gap condition in Corollary~\ref{cor:finite-spectral-approximation} becomes the lattice condition
\[
d_0\ne0,
\qquad
d_m=0
\quad\text{whenever}\quad
0<|m|/q\le 2B_0.
\]
Under this condition the finite spectral sums approximate the central term $d_0\widehat\varphi$, and the corresponding sampling formula uses the nodes $\lambda\in\Lambda_B$.

{
\section*{Acknowledgments}
This work was supported by the Swiss National Science Foundation, SNF grant No.~CRSK-2 229036. Part of this work was carried out during a research stay with the groups of Kristian Seip and Andrii Bondarenko. The authors thank Yves Meyer for his helpful feedback.
}

\FloatBarrier

\clearpage
\appsection{Coefficient estimates for powers of finite Blaschke products}{sec:asymBla}

This appendix collects the coefficient estimates used in Sections~\ref{sec:meyer-measures} and~\ref{sec:fourier-trunc-sampling}.
Appendix~\ref{app:meyer-bulk-lower} treats the one-factor case \(B=b_\lambda\),
\(f_0=1\), where \(c(q,n)=\widehat{b_\lambda^n}(q)\). Its lower bounds are used
in Section~\ref{sec:meyer-measures} for Meyer's example and for the sharpness
statement in Corollary~\ref{cor:sharp-three-halves}; Corollary~\ref{cor:l1-short}
also gives the matching lower bound to the \(\ell^1\)-row estimate
\eqref{eq:Blyudze}. Appendix~\ref{app:finite-blaschke-coeff-bounds} contains the estimates on the coefficients of finite Blaschke products needed in the general case.\par


\appsubsection{A lower bound for coefficients}{app:meyer-bulk-lower}

The precise input used below is the uniform stationary-phase expansion for coefficients of powers of a Blaschke factor on compact subintervals $I_\lambda$ of the oscillatory range $(\alpha_0, \alpha_0^{-1})$, where  $\alpha_0=\frac{1-\lambda}{1+\lambda}$.  
\begingroup
We use the same notation as in the end of Section~\ref{sec:meyer-measures} and in Theorem~\ref{Th_All_Regions}~(6) from Appendix~\ref{app:asymptotic-formulas}. Namely for \(t\in I_\lambda\) we set
\[
        \Phi_t(z)=\log\!\left(z^{-t}b_\lambda(z)\right),
        \qquad
        h_t(\varphi)=-i\Phi_t(e^{i\varphi})
        =\arg\!\left(e^{-it\varphi}b_\lambda(e^{i\varphi})\right),
        \qquad \varphi\in[0,\pi],
\]
with the same choice of branch as in Appendix~\ref{app:asymptotic-formulas}, and let
\(\varphi_+(t)\in(0,\pi)\) be the unique positive stationary point of \(h_t\). We then set
\[
        S_\lambda(t)=h_t\bigl(\varphi_+(t)\bigr),
        \qquad
        \Theta_n(q)=nS_\lambda(q/n)-\frac{\pi}{4}
        =n h_{q/n}\bigl(\varphi_+(q/n)\bigr)-\frac{\pi}{4}.
\]
\endgroup
Thus for \(q/n\in I_\lambda\),
\[
        \widehat{b_\lambda^n}(q)
        =n^{-1/2}A_\lambda(q/n)
        \bigl(\cos\Theta_n(q)+\varepsilon_n(q)\bigr),
        \qquad \varepsilon_n(q)\to0,
\]
where the convergence is uniform for \(q/n\in I_\lambda\); this is the stationary-phase
asymptotic from \cite[Theorem~2(2)]{BFZCA}, restricted to the fixed compact interval
\(I_\lambda\).  Such asymptotic analyses had already appeared in our earlier work:
for the coefficients of \(b_\lambda^n\) themselves, an Airy-transition
expansion near \(\alpha_0^{-1}n\) was obtained in
\cite[Proposition~6]{SzehrZaroufJAM}, while a closely related
stationary-phase analysis was developed in
\cite[Section~6]{SzehrZaroufSchaeffer2021} for the coefficients of
\((1-z^2)b_\lambda^n\). 
In the special case \(\lambda=1/2\), Meyer's argument invokes a stationary-phase lower estimate. Here we use the explicit cosine asymptotic with a uniform error term supplied by \cite{BFZCA},  which follows the techniques developed in \cite{SzehrZaroufJAM, SzehrZaroufSchaeffer2021}.  The factor \(A_\lambda\) is positive on
\(I_\lambda\), whereas \(\cos\Theta_n(q)\) may be close to zero for some values of
\(q\).  Thus the expansion does not yield a pointwise lower bound for every
coefficient in the interval.

The lemma below proves that there exists $L\geq2$ such that whenever the consecutive indices \(k,k+1,\ldots,k+L-1\) remain in \(nI_\lambda\), at least one of the corresponding coefficients has size comparable to \(n^{-1/2}\).  For
\(\lambda=1/2\), this is the phase-increment argument used in
\cite[Section~4, Eq.~(4.8)]{BFZCA}: the four consecutive indices
\(k,k+1,k+2,k+3\) are used to ensure that one of the corresponding cosine
factors is bounded away from zero, and therefore, after the uniform
stationary-phase error is made small, one of the four coefficients has size
at least a constant times \(n^{-1/2}\).  The same argument gives the following
statement for any \(\lambda\in(0,1)\).

\begin{lemma}\label{lem:osc-robust-lower}
Fix \(\lambda\in(0,1)\) and consider
\[
        b_\lambda(z)=\frac{z-\lambda}{1-\lambda z}.
\]
Let \(\psi_\lambda\) be defined on \((-\pi,\pi)\) by
\[
        b_\lambda(e^{i\theta})=e^{i\psi_\lambda(\theta)},
        \qquad \psi_\lambda(0)=0.
\]
For \(\theta\in(0,\pi)\),
\[
        \psi_\lambda'(\theta)
        =\frac{1-\lambda^2}{1+\lambda^2-2\lambda\cos\theta}.
\]
Thus \(\psi_\lambda'\) is strictly decreasing on \((0,\pi)\) with range
$\left(\alpha_0,\alpha_0^{-1}\right)$.      
Choose
\[
        0<c_-<c_+<\frac{2\pi}{3}
\]
and let \(L\ge2\) be an integer such that
\[
        (L-1)c_->\frac{\pi}{3}.
\]
Set
\[
        a_\lambda^-=\psi_\lambda'(c_+),
        \qquad
        a_\lambda^+=\psi_\lambda'(c_-),
        \qquad
        I_\lambda=[a_\lambda^-,a_\lambda^+].
\]
Then \(I_\lambda\subset(\alpha_0,\alpha_0^{-1})\).  There exist constants
\(c_0>0\) and \(n_0\ge1\), depending only on \(\lambda,c_-,c_+\), such that, for
every \(n\ge n_0\) and every integer \(k\) satisfying
\[
        \frac{k}{n}\in I_\lambda,
        \qquad
        k+L-1\le n a_\lambda^+,
\]
one has
\[
        \max_{0\le j\le L-1}
        \left|\widehat{b_\lambda^n}(k+j)\right|
        \ge c_0 n^{-1/2}.
\]
\end{lemma}

\begin{proof}
The coefficient has the oscillatory integral representation
\[
        \widehat{b_\lambda^n}(q)
        =\frac1{2\pi}\int_{-\pi}^{\pi}
        \exp\bigl(inh_{q/n}(\varphi)\bigr)\,d\varphi,
\]
where, on the branch fixed above,
\[
        h_t(\varphi)=\psi_\lambda(\varphi)-t\varphi,
        \qquad
        b_\lambda(e^{i\varphi})=e^{i\psi_\lambda(\varphi)}.
\]
The stationary points of \(h_t\) satisfy \(\psi_\lambda'(\varphi)=t\). Since
\(I_\lambda=[\psi_\lambda'(c_+),\psi_\lambda'(c_-)]\) and \(\psi_\lambda'\) is strictly
decreasing on \((0,\pi)\), every \(t\in I_\lambda\) has a unique positive stationary
point \(\varphi_+(t)\in[c_-,c_+]\). Moreover \(\psi_\lambda'\) is \(C^\infty\) and
\[
        (\psi_\lambda')'(\theta)
        =-\frac{2\lambda(1-\lambda^2)\sin\theta}
        {(1+\lambda^2-2\lambda\cos\theta)^2}
        <0,
        \qquad 0<\theta<\pi .
\]
Therefore the inverse function theorem shows that the map
\[
        \varphi_+:I_\lambda\longrightarrow[c_-,c_+]
\]
is \(C^1\). By the uniform stationary-phase expansion quoted above, uniformly for all integers \(q\) such that \(q/n\in I_\lambda\),
\[
        \widehat{b_\lambda^n}(q)
        =n^{-1/2}A_\lambda(q/n)
        \left(\cos\Theta_n(q)+\varepsilon_n(q)\right),
        \qquad \varepsilon_n(q)\to0,
\]
where \(A_\lambda\) is continuous and positive on \(I_\lambda\)
\begingroup and the phase is
\[
        \Theta_n(q)=nS_\lambda(q/n)-\frac{\pi}{4},
        \qquad
        S_\lambda(t)=h_t\bigl(\varphi_+(t)\bigr).
\]
Since \(\varphi_+\) is \(C^1\) we may differentiate \(S_\lambda\). Using
\(\partial_\varphi h_t(\varphi_+(t))=0\) and
\(\partial_t h_t(\varphi)=-\varphi\) we obtain
\[
        S_\lambda'(t)=-\varphi_+(t).
\]
\endgroup
The assumptions on \(k\) imply that \((k+j)/n\in I_\lambda\) for
\(0\le j\le L-1\).  Hence, for \(0\le j\le L-2\), the mean value theorem gives a point
\(\xi_j\in I_\lambda\) such that
\[
        \Theta_n(k+j+1)-\Theta_n(k+j)
        =n\left(S_\lambda\left(\frac{k+j+1}{n}\right)
              -S_\lambda\left(\frac{k+j}{n}\right)\right)
        =S_\lambda'(\xi_j)
        =-\varphi_+(\xi_j).
\]
Since \(\varphi_+(\xi_j)\in[c_-,c_+]\), we have
\[
        c_-\le \Theta_n(k+j)-\Theta_n(k+j+1)\le c_+<\frac{2\pi}{3}.
\]
Thus the phase decreases by a controlled amount at each successive index.

It remains to use only elementary trigonometry.  Put \(X_j=\Theta_n(k+j)\).  Suppose,
for a contradiction, that \(|\cos X_j|<1/2\) for every \(j=0,\ldots,L-1\).  Then each
\(X_j\) belongs to
\[
        E=\{x\in\mathbb R:\ |\cos x|<1/2\}
        =\bigcup_{m\in\mathbb Z}
        \left(\frac{\pi}{3}+m\pi,\frac{2\pi}{3}+m\pi\right).
\]
The intervals in this union have length \(\pi/3\), and two consecutive intervals are
separated by a distance \(2\pi/3\).  Since
\(0<X_j-X_{j+1}<2\pi/3\), two consecutive values \(X_j\) and \(X_{j+1}\) cannot belong
to two different intervals of the above union.  Consequently all the \(X_j\)'s belong to
a single interval of length \(\pi/3\).  This gives
\[
        X_0-X_{L-1}<\frac{\pi}{3}.
\]
On the other hand, summing the lower bound on the successive decreases gives
\[
        X_0-X_{L-1}
        =\sum_{j=0}^{L-2}(X_j-X_{j+1})
        \ge (L-1)c_->\frac{\pi}{3},
\]
which is impossible.  Therefore, for some \(j\in\{0,\ldots,L-1\}\),
\[
        |\cos\Theta_n(k+j)|\ge\frac12.
\]
Finally, let \(A_*=\min_{t\in I_\lambda}A_\lambda(t)>0\), and choose \(n_0\) so large
that \(|\varepsilon_n(q)|\le1/4\) whenever \(n\ge n_0\) and \(q/n\in I_\lambda\).  For
the index \(q=k+j\) found above,
\[
        \left|\widehat{b_\lambda^n}(q)\right|
        \ge n^{-1/2} A_*\left(\frac12-\frac14\right)
        =\frac{A_*}{4}n^{-1/2}.
\]
This proves the result, with \(c_0=A_*/4\).

\end{proof}
The following corollary is a consequence of Lemma~\ref{lem:osc-robust-lower}: it shows the sharpness of the upper estimate \eqref{eq:Blyudze}. Moreover it will be used in the proof of Corollary~\ref{cor:sharp-three-halves}.
\begin{corollary}\label{cor:l1-short}
For every \(\lambda\in(0,1)\), there are an interval
\([a_1,a_2]\subset(\alpha_0,\alpha_0^{-1})\) and a constant \(c(\lambda)>0\) such
that, for all sufficiently large \(n\),
\[
        \sum_{a_1n\le k\le a_2n}
        |\widehat{b_\lambda^n}(k)|
        \ge c(\lambda)\sqrt n.
\]
Consequently \(\|b_\lambda^n\|_{\ell^1}\ge c(\lambda) \sqrt n\) as \(n\to\infty\).
\end{corollary}

\begin{proof}
Set \(a_1=a_\lambda^-\) and \(a_2=a_\lambda^+\). Choose disjoint blocks of the form
\[
        \{k,k+1,\ldots,k+L-1\}
\]
inside the interval \([a_\lambda^- n,a_\lambda^+ n]\).  There are at least
\(c n\) such blocks, with \(c>0\) depending on \(\lambda,c_-,c_+\).  By
Lemma~\ref{lem:osc-robust-lower}, each block contains one coefficient with modulus at
least \(c_0n^{-1/2}\).  Summing over these blocks gives
\[
        \sum_{a_1n\le k\le a_2n}
        |\widehat{b_\lambda^n}(k)|
        \ge c(\lambda) n\cdot n^{-1/2}
        \ge c(\lambda) \sqrt n .
\]
The final statement follows because the left-hand side is bounded above by
\(\|b_\lambda^n\|_{\ell^1}\).
\end{proof}
\begin{corollary}\label{cor:sharp-three-halves}
Let \(\lambda\in(0,1)\), \(\alpha>0\), \(B=b_\lambda\), and \(f_0=1\). Then the
condition \(M>3/2\) in the weighted summability estimate of
Lemma~\ref{lem:weighted-summability-kappahatB} is sharp. More precisely,
\[
        \sum_{n\ge1}\sum_{k\in\mathbb Z}
        \left|\widehat{b_\lambda^n}(k)\right|
        (1+|k+\alpha n|)^{-M}
        =\infty
\]
for every \(M\le3/2\).
\end{corollary}

\begin{proof}
By Corollary~\ref{cor:l1-short}, there are numbers
\(0<a_1<a_2\) and a constant \(c=c(\lambda)>0\) such that, for all sufficiently large \(n\),
\[
        \sum_{a_1n\le k\le a_2n}
        |\widehat{b_\lambda^n}(k)|\ge c\sqrt n.
\]
For \(a_1n\le k\le a_2n\), we have
\[
        1+|k+\alpha n|\le C_{a_2,\alpha}(1+n).
\]
Therefore
\[
\begin{aligned}
        \sum_{k\in\mathbb Z}
        |\widehat{b_\lambda^n}(k)|
        (1+|k+\alpha n|)^{-M}
        &\ge
        C_{a_2,\alpha,M}^{-1}(1+n)^{-M}
        \sum_{a_1n\le k\le a_2n}
        |\widehat{b_\lambda^n}(k)|  \\
        &\ge C(1+n)^{1/2-M}.
\end{aligned}
\]
The series \(\sum_n (1+n)^{1/2-M}\) diverges exactly when \(M\le3/2\). This proves
the claim.
\end{proof}

\appsubsection{Bounds on the coefficients of powers of finite Blaschke products}{app:finite-blaschke-coeff-bounds}
We now return to the case of finite Blaschke products \(B\).    We keep the notation of the
main text: if
\[
        f_n=B^n f_0,
        \qquad
        c(k,n)=\widehat{f_n}(k),
\]
then, for \(f_0\equiv1\),
\[
        c(k,n)=\widehat{B^n}(k).
\]
If \(\widehat{f_0}\in\ell^1(\mathbb Z)\), i.e. if $f_0$ belongs to the Wiener algebra $W=\{f:\widehat f\in\ell^1(\mathbb Z)\}$, multiplication on the circle gives
\[
        c(\cdot,n)=\widehat{B^n}*\widehat{f_0}.
\]
Thus
\[
        \sup_{k\in\mathbb Z}|c(k,n)|
        \le
        \|\widehat{f_0}\|_{\ell^1}\,
        \|\widehat{B^n}\|_{\ell^\infty}.
\]
By \cite[Theorem~1]{BFZIMRN}, this gives
\[
        \sup_{k\in\mathbb Z}|c(k,n)|
        \leq
        \ C_B\|\widehat{f_0}\|_{\ell^1}\,n^{-1/N_B},
\]
where $C_B>0$ and \(N_B\) is defined as follows.  If
\[
        B(e^{i\theta})=\exp(i\psi_B(\theta)),
        \qquad
        \psi_B(\theta+2\pi)=\psi_B(\theta)+2\pi\deg(B),
\]
then \(N_B\ge3\) is the maximal order of degeneracy of the zeros of
\(\psi_B''\), in the sense of \cite{BFZIMRN}.  Moreover,
\cite[Theorem~2]{BFZIMRN} shows that this exponent cannot be replaced by one exponent valid for all
finite Blaschke products: for every integer \(N\ge3\), there exists a finite Blaschke product
\(B\) such that
\[
        c_B n^{-1/N} \leq \|\widehat{B^n}\|_{\ell^\infty}\leq C_B n^{-1/N}, 
\]
where $c_B>0$. This is why the estimate is stated with the exponent \(N_B\).

We also need an exponentially decaying estimate when \(k/n\) is away from the interval where the coefficients can be large.  For a single Blaschke factor, the corresponding is result is stated and proved in
\cite[Proposition~3(1)]{SzehrZaroufJAM}.  The next proposition gives the same kind of estimate for finite Blaschke products.  We use the interval \([m_B,M_B]\), defined
in Theorem~\ref{thm:general-meyer}, as a simple range for the ratios \(k/n\).  If \(k/n\) stays in a
compact interval disjoint from this range, then \(\widehat{B^n}(k)\) decreases
exponentially in \(|n|\).

This estimate is used in Section~\ref{sec:meyer-measures}, in particular in the
proof of Theorem~\ref{thm:general-meyer}, through
Lemma~\ref{lem:radon-tempered-projected-tail}.  We include the proof because it
only uses the formula for Fourier coefficients as contour integrals on circles of
radius different from \(1\).

\begin{prop}\label{prop:exp-tails-Bn}
Let $B$ be a finite Blaschke product with zeros $\lambda_1,\ldots,\lambda_d$, counted with multiplicity, and set
\[
        m_B=\sum_{j=1}^d\frac{1-|\lambda_j|^2}{(1+|\lambda_j|)^2},
        \qquad
        M_B=\sum_{j=1}^d\frac{1-|\lambda_j|^2}{(1-|\lambda_j|)^2}.
\]
For every compact interval $J\subset\mathbb R\setminus[m_B,M_B]$, there exists a constant $\rho_J>0$ such that
\[
        |\widehat{B^n}(k)|\le e^{-\rho_J |n|}
\]
whenever $n\neq0$ and $k/n\in J$.
\end{prop}

\begin{proof}
Write
\[
        B(z)=\prod_{j=1}^d b_{\lambda_j}(z),
        \qquad
        b_{\lambda_j}(z)=\frac{z-\lambda_j}{1-\overline{\lambda_j}z}.
\]
and put \(a_j=|\lambda_j|\).  

We first estimate the coefficients on the left of the interval \([m_B,M_B]\).  For
\(0<r<1\), the elementary estimate for one Blaschke factor gives
\[
        \max_{|z|=r}|b_{\lambda_j}(z)|
        \leq \frac{r+a_j}{1+a_jr},
\]
see Garnett~\cite{Garnett1981}.  This is also the estimate used in
\cite[Proposition~3(1)]{SzehrZaroufJAM} for a single factor.  Hence
\[
        \max_{|z|=r}|B(z)|
        \leq L_-(r)=\prod_{j=1}^d\frac{r+a_j}{1+a_jr}.
\]
Moreover
\[
        \lim_{r\to1^-}\frac{\log L_-(r)}{\log r}
        =\sum_{j=1}^d\frac{1-a_j}{1+a_j}
        =m_B.
\]
Thus, if \(\sigma<m_B\), we may choose \(r<1\), close enough to \(1\), such that
\[
        L_-(r)r^{-\sigma}<1.
\]
Indeed, \(\log r<0\), so this follows directly from the preceding limit.

We now estimate the right tail.  Choose \(R>1\), close enough to \(1\), so that \(B\) is holomorphic on \(|z|\leq R\).  For one factor, a direct computation gives
\[
        \max_{|z|=R}|b_{\lambda_j}(z)|
        =\frac{R-a_j}{1-a_jR}.
\]
Hence
\[
        \max_{|z|=R}|B(z)|
        \leq L_+(R)=\prod_{j=1}^d\frac{R-a_j}{1-a_jR}.
\]
Also
\[
        \lim_{R\to1^+}\frac{\log L_+(R)}{\log R}
        =\sum_{j=1}^d\frac{1+a_j}{1-a_j}
        =M_B.
\]
Therefore, if \(\sigma>M_B\), we may choose \(R>1\), close enough to \(1\), such that
\[
        L_+(R)R^{-\sigma}<1.
\]

Let \(J\subset\mathbb R\setminus[m_B,M_B]\) be compact.  Choose
\(\sigma_-<m_B<M_B<\sigma_+\) such that
\[
        J\subset(-\infty,\sigma_-]\cup[\sigma_+,+\infty).
\]
Then choose \(r<1<R\) close enough to 1 so that
\[
        L_-(r)r^{-\sigma_-}<1,
        \qquad
        L_+(R)R^{-\sigma_+}<1.
\]
Set
\[
        \rho_-=-\log\bigl(L_-(r)r^{-\sigma_-}\bigr)>0,
        \qquad
        \rho_+=-\log\bigl(L_+(R)R^{-\sigma_+}\bigr)>0.
\]

Assume first that \(n\geq1\).  Since \(B^n\) is holomorphic in the unit disk,
\[
        \widehat{B^n}(k)=0\qquad, k<0.
\]
For \(k\geq0\), the contour formula on \(|z|=r\) gives
\[
        |\widehat{B^n}(k)|
        \leq L_-(r)^n r^{-k}.
\]
Hence if \(0\leq k\leq \sigma_- n\) then
\[
        |\widehat{B^n}(k)|
        \leq \bigl(L_-(r)r^{-\sigma_-}\bigr)^n
        =e^{-\rho_-n}.
\]
Similarly the contour formula on \(|z|=R\) gives
\[
        |\widehat{B^n}(k)|
        \leq L_+(R)^n R^{-k}.
\]
Hence if \(k\geq\sigma_+n\) then
\[
        |\widehat{B^n}(k)|
        \leq \bigl(L_+(R)R^{-\sigma_+}\bigr)^n
        =e^{-\rho_+n}.
\]
This proves the required estimate for positive powers. For negative powers, write \(n=-m\) with \(m\ge1\). Since
\(B^{-m}=\overline{B^m}\) on \(\mathbb T\), conjugating the Fourier expansion of
\(B^m\) gives
\[
\widehat{B^{-m}}(k)=0\quad(k>0),
\qquad
\widehat{B^{-m}}(k)=\overline{\widehat{B^m}(-k)}\quad(k\le0).
\]
The estimate for \(B^m\) therefore gives the same
estimate for \(B^{-m}\).  Taking
\[
        \rho_J=\min(\rho_-,\rho_+)
\]
completes the proof.
\end{proof}


%
%
%
\appsection{Asymptotic formulas for powers of a Blaschke factor}{app:asymptotic-formulas}
The results below are taken from \cite{BFZCA}. Throughout this appendix we assume without loss of generality that $\lambda\in(0,1)$. Here \(n\ge 1\) denotes the power and \(k\in\mathbb Z_+\) denotes the Fourier coefficient index in \(\widehat{b_\lambda^n}(k)\).
\begingroup
In order to state the asymptotic formulas on these coefficients we use the following notation.
For quantities \(F\) and \(G\) with \(G\neq0\) we write
\(F\sim G\) if \(F/G\to1\) as \(n\to\infty\) always within the range of $k$ specified in the corresponding statement. For nonnegative quantities
\(F\) and \(G\), we write \(F\asymp G\) if there exist constants \(c,C>0\), independent of \(n\) and \(k\) in the range under consideration, such that \(cG\le F\le CG\).
\endgroup

For real arguments $x$ the Airy function $Ai(x)$ can be defined
as an improper Riemann integral
\[
Ai(x)=\frac{1}{\pi}\int_{0}^{\infty}\cos\left(\frac{t^{3}}{3}+xt\right)\,\dd t.
\]

\begingroup
We follow the notation of \cite{BFZCA}. Whenever a remainder term appears
in the approximations of \(\widehat{b_\lambda^n}(k)\), we keep its dependence on \(k\) visible in the notation. This is especially the case near the \(k\)-transition points \(n(1-\lambda)/(1+\lambda)\) and \(n(1+\lambda)/(1-\lambda)\). Since the error terms depend on the coefficient index \(k\), we introduce \(Ai_*^{(n,k)}(x)\), where \(x\) denotes the Airy variable
obtained by measuring the distance from \(k\) to the relevant transition point. We set
\[
Ai_*^{(n,k)}(x)
=
\begin{cases}
Ai(x)\bigl(1+\varepsilon_{n,1}(k)\bigr)
+(1+|x|)^{-1/4}\varepsilon_{n,2}(k), & x<0,\\[0.4em]
Ai(x)\bigl(1+\varepsilon_{n,1}(k)\bigr), & x\ge 0,
\end{cases}
\]
where the error terms \(\varepsilon_{n,1}(k)\) and
\(\varepsilon_{n,2}(k)\) tend to \(0\) uniformly for \(k\) in the corresponding transition region.
\endgroup

{Below $\alpha_{0}$ is given by $\alpha_{0}=\frac{1-\lambda}{1+\lambda}$. For each ratio \(t=k/n\) occurring below we set}
\[
\Phi(z)=\Phi_{k/n}(z)=\log\left(z^{-\frac{k}{n}}b_{\lambda}(z)\right),
\]
where $\log$ denotes a branch of the complex logarithm chosen in
the following way: if $k/n\le c<\alpha_{0}^{-1}$, then we can take
the branch cut $[0,\infty)$ and fix $\log(-1)=i\pi$, and if $k/n\ge c>\alpha_{0}$,
then we can take the principal branch of the complex logarithm. In
particular, if $\alpha_{0}<c_{1}\le k/n\le c_{2}<\alpha_{0}^{-1}$,
then we could take either of these two definitions. We also define
the function $h$ on $[0,\pi]$ as follows:
\[
h(\varphi)=h_{k/n}(\varphi)=-i\Phi_{k/n}(e^{i\varphi})=\arg\left(\left(z^{-k/n}b_{\lambda}(z)\right)_{\vert z=e^{i\varphi}}\right)\qquad\varphi\in[0,\pi].
\]
\begin{thm}
\label{Th_All_Regions} \label{Exponentiel decay of the Fourier coefficients of the powers of a Blaschke factor}
Let $\alpha\in(0,\alpha_{0})$. Consider a sequence $\omega(n^{1/3})$
such that $\omega(n^{1/3})/n^{1/3}\rightarrow\infty$ as $n\rightarrow\infty$
and assume additionally that $\omega(n^{1/3})=o(n)$ as $n\rightarrow\infty$.
The following asymptotic formulas for the $k^{{\rm th}}-$Fourier
coefficients of $b_{\lambda}^{n}$ hold as $n$ tends to $+\infty$.

(1) If $k$ is fixed (Region I), then
\[
\widehat{b_{\lambda}^{n}}(k)\sim\frac{(-\lambda)^{n-k}\left(n(1-\lambda^{2})\right)^{k}}{k!}.
\]

(2) If $k=k(n)\rightarrow\infty$ as $n\rightarrow\infty$ with $k\leq\alpha n$
(Region II) or $k\geq\alpha^{-1}n$ (Region VIII), then
\[
\widehat{b_{\lambda}^{n}}(k)\sim\frac{1}{\sqrt{2k\pi}}\frac{1}{\left[(\alpha_{0}-k/n)(\alpha_{0}^{-1}-k/n)\right]^{1/4}}\left(\frac{b_{\lambda}(z_{+})}{z_{+}^{k/n}}\right)^{n},
\]
where $z_{+}$ is defined by
\[
z_{+}=z_{+}(k/n)=\frac{\frac{k}{n}(1+\lambda^{2})-(1-\lambda^{2})}{2\lambda\frac{k}{n}}+\sqrt{\left(\frac{\frac{k}{n}(1+\lambda^{2})-(1-\lambda^{2})}{2\lambda\frac{k}{n}}\right)^{2}-1}.
\]

(3) If $k\in[\alpha n,\alpha_{0}n-\omega(n^{1/3})]$ (Region III),
then
\[
\widehat{b_{\lambda}^{n}}(k)\sim\frac{(-1)^{n-k}}{\sqrt{2n\pi}}\frac{1}{\sqrt{k/n}\left[(\alpha_{0}^{-1}-k/n)(\alpha_{0}-k/n)\right]^{1/4}}\exp\left(-\frac{2}{3}n\abs{\gamma_{{\alpha_{0}}}}^{3}\right),
\]
where $\gamma_{\alpha_{0}}^{3}$ is given by
\[
\gamma_{\alpha_{0}}^{3}=\frac{3}{2}\left[\Phi(z_{+})-i\pi\left(1-\frac{k}{n}\right)\right],
\]
and in particular
\[
\gamma_{\alpha_{0}}^{3}\sim-\frac{\left(\alpha_{0}-k/n\right)^{3/2}(1+\lambda)^{3/2}}{(\lambda(1-\lambda))^{1/2}},\qquad k/n\to\alpha_{0},\qquad k/n<\alpha_{0}.
\]

(4) If $k\in[\alpha_{0}^{-1}n+\omega(n^{1/3}),\alpha^{-1}n]$ (Region
VII), then
\[
\widehat{b_{\lambda}^{n}}(k)\sim\frac{1}{\sqrt{2n\pi}}\frac{1}{\sqrt{k/n}\left[(k/n-\alpha_{0}^{-1})(k/n-\alpha_{0})\right]^{1/4}}\exp\left(-\frac{2}{3}n\abs{\gamma_{{\alpha_{0}}^{-1}}}^{3}\right),
\]
where $\gamma_{\alpha_{0}^{-1}}^{3}$ is given by
\[
\gamma_{\alpha_{0}^{-1}}^{3}=\frac{3}{2}\Phi(z_{+}),
\]
and in particular
\[
\gamma_{\alpha_{0}^{-1}}^{3}\sim-\frac{(k/n-\alpha_{0}^{-1})^{3/2}(1-\lambda)^{3/2}}{\left(\lambda(1+\lambda)\right)^{1/2}},\qquad k/n\to\alpha_{0}^{-1},\qquad k/n>\alpha_{0}^{-1}.
\]

(5) If $k\in[\alpha_{0}n-\omega(n^{1/3}),\alpha_{0}n+\omega(n^{1/3})]$
(Region IV), then
\[
\widehat{b_{\lambda}^{n}}(k)=(-1)^{n-k}\frac{1+\lambda}{\bigl(\lambda(1-\lambda)n\bigr)^{1/3}}\,
Ai_*^{(n,k)}\left(n^{2/3}\gamma_{\alpha_{0}}^{2}\right),
\]
where $\gamma_{\alpha_{0}}^{2}$ is asymptotically given by
\[
\gamma_{\alpha_{0}}^{2}\sim\frac{\left(\alpha_{0}-k/n\right)(1+\lambda)}{(\lambda(1-\lambda))^{1/3}},\qquad k/n\to\alpha_{0}.
\]

(6) If $k\in[\alpha_{0}n+\omega(n^{1/3}),\alpha_{0}^{-1}n-\omega(n^{1/3})]$
(Region V), then
\[
\widehat{b_{\lambda}^{n}}(k)=\sqrt{\frac{2}{n\pi}}\frac{\cos\left(nh(\varphi_{+})-\pi/4\right)+\varepsilon_n(k)}{\sqrt{k/n}\left[(\alpha_{0}^{-1}-k/n)(k/n-\alpha_{0})\right]^{1/4}},
\]
where $h=h_{k/n}$ is the function defined above and the parameter $\varphi_{+}\in[0,\pi]$
is defined by
\[
e^{i\varphi_{+}}=z_{+}=\frac{\frac{k}{n}(1+\lambda^{2})-(1-\lambda^{2})}{2\lambda\frac{k}{n}}+i\sqrt{1-\left(\frac{\frac{k}{n}(1+\lambda^{2})-(1-\lambda^{2})}{2\lambda\frac{k}{n}}\right)^{2}}.
\]
Here the error term \(\varepsilon_{n}(k)\) tends to \(0\) uniformly for \(k\) in Region V.

(7) If $k\in[\alpha_{0}^{-1}n-\omega(n^{1/3}),\alpha_{0}^{-1}n+\omega(n^{1/3})]$
(Region VI), then
\[
\widehat{b_{\lambda}^{n}}(k)=\frac{1-\lambda}{\bigl(\lambda(1+\lambda)n\bigr)^{1/3}}\,
Ai_*^{(n,k)}\left(n^{2/3}\gamma_{\alpha_{0}^{-1}}^{2}\right),
\]
where
\[
\gamma_{\alpha_{0}^{-1}}^{2}\sim\frac{(k/n-\alpha_{0}^{-1})(1-\lambda)}{\left(\lambda(1+\lambda)\right)^{1/3}},\qquad k/n\to\alpha_{0}^{-1}.
\]
\end{thm}
\newpage
\subsection*{Table summing up the asymptotic behavior of the coefficients}

\begin{center}
\captionsetup{type=table}
\caption{Asymptotic regimes for $\widehat{b_{\lambda}^{n}}(k)$ as $n\to\infty$.}
\label{tab:asymptotics}
\resizebox{\textwidth}{!}{
\begin{tabular}{|c|c|c|}
\hline
Values of $k(n)$ in interval  & Asymptotics of $\widehat{b_{\lambda}^{n}}(k)$  & Region\tabularnewline
\hline
\hline
$(0,\,\alpha n]$  & $\frac{1}{\sqrt{k/n}\left[(\alpha_{0}-k/n)(\alpha_{0}^{-1}-k/n)\right]^{1/4}}\frac{1}{\sqrt{n}}\left(\frac{b_{\lambda}(z_{+})}{z_{+}^{k/n}}\right)^{n}$  & I-II\tabularnewline
\hline
$(\alpha n,\,\alpha_{0}n-\omega(n^{1/3})]$  & $\frac{1}{\sqrt{k/n}\left[(\alpha_{0}^{-1}-k/n)(\alpha_{0}-k/n)\right]^{1/4}}\text{\ensuremath{\frac{1}{\sqrt{n}}}}e^{-\frac{2}{3}n\abs{\gamma_{{\alpha_{0}}}}^{3}}$  & III\tabularnewline
\hline
$[\alpha_{0}n-\omega(n^{1/3}),\,\alpha_{0}n+\omega(n^{1/3})]$  & $\frac{1}{\sqrt{k/n}(\alpha_{0}^{-1}-k/n)^{1/4}}\frac{Ai_*^{(n,k)}(n^{2/3}\gamma_{{\alpha_{0}}}^{2})}{n^{1/3}}$  & IV\tabularnewline
\hline
$[\alpha_{0}n+\omega(n^{1/3}),\,\alpha_{0}^{-1}n-\omega(n^{1/3})]$  & $\frac{1}{\sqrt{n}}\frac{\cos\left(nh(\varphi_{+})-\pi/4\right)+\varepsilon_n(k)}{\sqrt{k/n}\left[(\alpha_{0}^{-1}-k/n)(k/n-\alpha_{0})\right]^{1/4}}$  & V\tabularnewline
\hline
$[\alpha_{0}^{-1}n-\omega(n^{1/3}),\,\alpha_{0}^{-1}n+\omega(n^{1/3})]$  & $\frac{1}{\sqrt{k/n}(k/n-\alpha_{0})^{1/4}}\frac{Ai_*^{(n,k)}(n^{2/3}\gamma_{{\alpha_{0}}^{-1}}^{2})}{n^{1/3}}$  & VI\tabularnewline
\hline
$[\alpha_{0}^{-1}n+\omega(n^{1/3}),\,\alpha^{-1}n)$  & $\frac{1}{\sqrt{k/n}\left[(k/n-\alpha_{0}^{-1})(k/n-\alpha_{0})\right]^{1/4}}\frac{1}{\sqrt{n}}e^{-\frac{2}{3}n\abs{\gamma_{{\alpha_{0}}^{-1}}}^{3}}$  & VII\tabularnewline
\hline
$[\alpha^{-1}n,\,\infty)$  & $\frac{1}{\sqrt{k/n}\left[(\alpha_{0}-k/n)(\alpha_{0}^{-1}-k/n)\right]^{1/4}}\frac{1}{\sqrt{n}}\left(\frac{b_{\lambda}(z_{+})}{z_{+}^{k/n}}\right)^{n}$  & VIII\tabularnewline
\hline
\end{tabular}
}
\caption*{Asymptotic formulas for $\widehat{b_{\lambda}^{n}}(k)$ as $n\rightarrow\infty$,
up to numerical factors. For $k$ in Regions I--II and VIII, we have
$\abs{z_{+}^{-k/n}b_{\lambda}(z_{+})}<1$ and the decay of $\widehat{b_{\lambda}^{n}}(k)$
is exponential. The values $\gamma_{\alpha_{0}}$ and $\gamma_{{\alpha_{0}}^{-1}}$
are asymptotically given by $\gamma_{\alpha_{0}}^{2}\asymp\alpha_{0}-k/n$
and $\gamma_{\alpha_{0}^{-1}}^{2}\asymp k/n-\alpha_{0}^{-1}$ respectively
as $k/n\to\alpha_{0}$ and $\alpha_{0}^{-1}$. The formulas for $k$
in Regions III and VII ensure the transition between the exponential
decay (Regions I--II and VIII) and the $\cO(n^{-1/3})$ decay, which
occurs in Regions IV and VI when the distance between $k$ and $\alpha_{0}n$
respectively $\alpha_{0}^{-1}n$ does not exceed $n^{1/3}$. Finally,
the formula for $k$ in Region V ensures the transition to an oscillatory
decay of order $\cO(n^{-1/2})$ when $k$ is away from the boundaries
$\alpha_{0}n$ and $\alpha_{0}^{-1}n$ (we refer to Theorem \ref{Th_All_Regions}~(6)
for the definition of $h(\varphi_{+})$).}
\end{center}


\clearpage

\newpage

\clearpage

\begin{thebibliography}{99}

\bibitem{BaakeGrimm2012}
M. Baake and U. Grimm,
\emph{Mathematical diffraction of aperiodic structures},
Chem. Soc. Rev. 41 (2012), 6821--6843.

\bibitem{BaakeStrungaruTerauds2020}
M. Baake, N. Strungaru, and V. Terauds,
\emph{Pure point measures with sparse support and sparse Fourier--Bohr support},
Trans. London Math. Soc. 7 (2020), no. 1, 1--48.

\bibitem{BaranovJamingKellaySpeckbacher2024}
{A.~D.~Baranov, P.~Jaming, K.~Kellay, and M.~Speckbacher,
\emph{Oversampling and Donoho--Logan type theorems in model spaces},
Ann. Fenn. Math. 49 (2024), no.~1, 167--182.}

\bibitem{BlyudzeShimorin}
M. Yu. Blyudze and S. M. Shimorin,
\emph{Estimates of the norms of powers of functions in certain Banach spaces},
J. Math. Sci. 80 (1996), no. 4, 1880--1891.

\bibitem{BondarenkoRadchenkoSeip2023}
A. Bondarenko, D. Radchenko, and K. Seip,
\emph{Fourier interpolation with zeros of zeta and L-functions},
Constr. Approx. 57 (2023), no. 2, 405--461.

\bibitem{BFZCA}
A. Borichev, K. Fouchet, and R. Zarouf,
\emph{On the Fourier coefficients of powers of a Blaschke factor and strongly annular functions},
Constr. Approx. 60 (2024), 33--86.

\bibitem{BFZIMRN}
A. Borichev, K. Fouchet, and R. Zarouf,
\emph{On the Fourier coefficients of powers of a finite Blaschke product},
Int. Math. Res. Not. IMRN 2024 (2024), no. 20, 13255--13280.

\bibitem{Bourbaki}
N. Bourbaki,
\emph{Integration. Chapters 1--6},
Elements of Mathematics, Springer, Berlin, 2004.

\bibitem{Clark1972}
D. N. Clark,
\emph{One-dimensional perturbations of restricted shifts},
J. Anal. Math. 25 (1972), 169--191.

\bibitem{Cordoba1989}
A. C\'ordoba,
\emph{Dirac combs},
Lett. Math. Phys. 17 (1989), 191--196.


\bibitem{Goncalves2026}
{F.~Gon\c{c}alves,
\emph{A classification of Fourier summation formulas and crystalline measures},
arXiv:2312.11185, 2026.}

\bibitem{Garnett1981} 
J.~Garnett, \emph{Bounded Analytic Functions}, Academic Press, New York, 1981.

\bibitem{Guinand1959}
A. P. Guinand,
\emph{Concordance and the harmonic analysis of sequences},
Acta Math. 101 (1959), no. 3--4, 235--271.

\bibitem{HartmannJamingKellay2020}
{A.~Hartmann, P.~Jaming, and K.~Kellay,
\emph{Quantitative estimates of sampling constants in model spaces},
Amer. J. Math. 142 (2020), no.~4, 1301--1326.}

\bibitem{HormanderI}
L. H\"ormander,
\emph{The Analysis of Linear Partial Differential Operators I},
2nd ed., Springer, Berlin, 1990.

\bibitem{Jagerman1966}
D. Jagerman,
\emph{Bounds for truncation error of the sampling expansion},
SIAM J. Appl. Math. 14 (1966), no. 4, 714--723.

\bibitem{Jerri1977}
A. J. Jerri,
\emph{The Shannon sampling theorem---its various extensions and applications: A tutorial review},
Proc. IEEE 65 (1977), no. 11, 1565--1596.

\bibitem{KahaneMandelbrojt1958}
J.-P. Kahane and S. Mandelbrojt,
\emph{Sur l'\'equation fonctionnelle de Riemann et la formule sommatoire de Poisson},
Ann. Sci. \'Ec. Norm. Sup\'er. (3) 75 (1958), no. 1, 57--80.

\bibitem{KayaalpSzehr2026}
M. Kayaalp and O. Szehr,
\emph{Signal recovery from time and frequency samples},
arXiv:2603.16242, 2026.

\bibitem{Kolountzakis2016}
M. N. Kolountzakis,
\emph{Fourier pairs of discrete support with little structure},
J. Fourier Anal. Appl. 22 (2016), no. 1, 1--5.

\bibitem{Kulikov2021}
A. Kulikov,
\emph{Fourier interpolation and time-frequency localization},
J. Fourier Anal. Appl. 27 (2021), no. 3, Paper No. 58, 8 pp.

\bibitem{KNS2023}
A. Kulikov, F. Nazarov, and M. Sodin,
\emph{Fourier uniqueness and non-uniqueness pairs},
J. Math. Phys. Anal. Geom. 21 (2025), no.~1, 84--130.

\bibitem{KurasovSarnak2020}
{P.~Kurasov and P.~Sarnak,
\emph{Stable polynomials and crystalline measures},
J. Math. Phys. 61 (2020), no.~8, 083501.}

\bibitem{LevOlevskii2013}
N. Lev and A. Olevskii,
\emph{Measures with uniformly discrete support and spectrum},
C. R. Math. Acad. Sci. Paris 351 (2013), no. 15--16, 599--603.

\bibitem{LevOlevskii2015}
N. Lev and A. Olevskii,
\emph{Quasicrystals and Poisson's summation formula},
Invent. Math. 200 (2015), no. 2, 585--606.

\bibitem{LevOlevskii2016}
N. Lev and A. Olevskii,
\emph{Quasicrystals with discrete support and spectrum},
Rev. Mat. Iberoam. 32 (2016), no. 4, 1341--1352.


\bibitem{MazacRichardStrungaru2026}
{J.~Maz\'a\v{c}, C.~Richard, and N.~Strungaru,
\emph{On almost periodicity in crystalline measures},
arXiv:2605.23884, 2026.}


\bibitem{MeyerPNAS16}
Y. F. Meyer,
\emph{Measures with locally finite support and spectrum},
Proc. Natl. Acad. Sci. USA 113 (2016), no. 12, 3152--3158.

\bibitem{MeyerRMI2017}
Y. F. Meyer,
\emph{Measures with locally finite support and spectrum},
Rev. Mat. Iberoam. 33 (2017), no. 3, 1025--1036.

\bibitem{Meyer}
Y. F. Meyer,
\emph{Crystalline measures and mean-periodic functions},
DKNVS Skrifter 2 (2021), 5--30.

\bibitem{MeyerBourbaki1194}
Y. F. Meyer,
\emph{Mesures cristallines et applications},
S\'eminaire Bourbaki, Ast\'erisque 438 (2022), Exp. No. 1194, 479--494.

\bibitem{MeyerMultidimensional2023}
{Y.~F.~Meyer,
\emph{Multidimensional crystalline measures},
Trans. R. Norw. Soc. Sci. Lett. 2023 (2023), no.~1, 1--24.}


\bibitem{MeyerKahane2026}
Y. F. Meyer,
\emph{En relisant Jean-Pierre Kahane},
to appear in a volume in honor of Jean-Pierre Kahane, 2026.

\bibitem{MeyerKolmogorov}
{Y.~F.~Meyer,
\emph{Independent random variables and classical analysis},
forthcoming.}

\bibitem{MeyerPrivate2026}
{Y.~F.~Meyer,
\emph{Private communication}, 2026.}

\bibitem{NikolskisBook}
N. Nikolski,
\emph{Treatise on the Shift Operator}, Translated from the Russian \emph{Lektsii ob Operatore Sdviga}, "Nauka", Moscow, 1980. Springer-Verlag, Berlin, 1986.

\bibitem{RadchenkoViazovska2019}
D. Radchenko and M. Viazovska,
\emph{Fourier interpolation on the real line},
Publ. Math. Inst. Hautes \'Etudes Sci. 129 (2019), 51--81.

\bibitem{RamosSousa2022}
J. P. G. Ramos and M. Sousa,
\emph{Fourier uniqueness pairs of powers of integers},
J. Eur. Math. Soc. (JEMS) 24 (2022), no. 12, 4327--4351.

\bibitem{Saksman2007}
E. Saksman,
\emph{An elementary introduction to Clark measures},
in Topics in Complex Analysis and Operator Theory,
Universidad de M\'alaga, M\'alaga, 2007, pp. 85--136.

\bibitem{Szehr2025}
O. Szehr,
\emph{Spectral criteria for unique signal recovery from two-sided sampling},
arXiv:2509.14953, 2025.

\bibitem{SzehrZaroufJAM}
O. Szehr and R. Zarouf,
\emph{\(\ell^p\)-norms of Fourier coefficients of powers of a Blaschke factor},
J. Anal. Math. 140 (2020), 1--30.

\bibitem{SzehrZaroufSchaeffer2021}
O. Szehr and R. Zarouf,
\emph{Explicit counterexamples to Sch\"affer's conjecture},
J. Math. Pures Appl. 146 (2021), 1--30.

\bibitem{SzehrZaroufJAT2022}
O. Szehr and R. Zarouf,
\emph{On the asymptotic behavior of Jacobi polynomials with first varying parameter},
J. Approx. Theory 277 (2022), Paper No. 105702.


\bibitem{Weil1952}
A. Weil,
\emph{Sur les "formules explicites" de la th\'eorie des nombres premiers},
Comm. S\'em. Math. Univ. Lund [Medd. Lunds Univ. Mat. Sem.] (1952),
Tome suppl\'ementaire, 252--265.
\end{thebibliography}
\end{document}